\documentclass[reqno,11pt]{amsart}

\usepackage{graphicx}

\usepackage[ansinew]{inputenc}
\usepackage{amsfonts,epsfig}
\usepackage{latexsym}
\usepackage{mathabx}
\usepackage{amsmath}
\usepackage{amssymb}

\usepackage{color}

\newtheorem{theorem}{Theorem}
\newtheorem{lemma}[theorem]{Lemma}

\newtheorem{conjecture}[theorem]{Conjecture}
\newtheorem{proposition}[theorem]{Proposition}

\theoremstyle{definition}

\theoremstyle{remark}

\numberwithin{equation}{section}

\newcommand{\B}{\mathcal{B}}
\newcommand{\D}{\mathbb{D}}
\newcommand{\DD}{\widehat{\mathcal{D}}}
\newcommand{\Dd}{\widecheck{\mathcal{D}}}

\newcommand{\DDD}{\mathcal{D}}

\newcommand{\N}{\mathbb{N}}

\newcommand{\RR}{\mathcal{R}}

\newcommand{\C}{\mathbb{C}}

\newcommand{\e}{\varepsilon}

\renewcommand{\phi}{\varphi}

\newcommand{\whw}{\widehat{\omega}}

\def\a{\alpha}               
           \def\e{\varepsilon}
     \def\om{\omega}      
              
                  \def\z{\zeta}

\renewcommand{\H}{\mathcal{H}}

\begin{document}
\title[Two weight inequality for Hankel form on weighted Bergman spaces]{Two weight inequality for Hankel form on weighted Bergman spaces induced by doubling weights}


\thanks{The research is supported in part by NNSF of China (Grant Number: 12171075), Science and technology development plan of Jilin Province (Grant Number: 20210509039RQ) and the University of Eastern Finland for financial support.}

\begin{abstract}
The boundedness of the small Hankel operator $h_f^\nu(g)=P_\nu(f\overline{g})$, induced by an analytic symbol $f$ and the Bergman projection $P_\nu$ associated to $\nu$, acting from the weighted Bergman space $A^p_\om$ to $A^q_\nu$ is characterized on the full range $0<p,q<\infty$ when $\omega,\nu$ belong to the class $\mathcal{D}$ of radial weights admitting certain two-sided doubling conditions. Certain results obtained are equivalent to the boundedness of bilinear Hankel forms, which are in turn used to establish the weak factorization $A_{\eta}^{q}=A_{\omega}^{p_{1}}\odot A_{\nu}^{p_{2}}$, where $1<q,p_{1},p_{2}<\infty$ such that $q^{-1}=p_{1}^{-1}+p_{2}^{-1}$ and $\widetilde{\eta}^{\frac{1}{q}}\asymp\widetilde{\omega}^{\frac{1}{p_{1}}}\widetilde{\nu}^{\frac{1}{p_{2}}}$. Here $\widetilde{\tau}(r)=\int_r^1\tau(t)\,dt/(1-t)$ for all $0\le r<1$.
\end{abstract}



\author[Yongjiang Duan]{Yongjiang Duan}
\address{School of Mathematics and Statistics\\
Northeast Normal University\\ Changchun\\ Jilin 130024\\ P.R.China}
\email{duanyj086@nenu.edu.cn}

\author[Jouni R\"atty\"a]{Jouni R\"atty\"a}
\address{University of Eastern Finland\\
Department of Physics and Mathematics\\
P.O.Box 111\\FI-80101 Joensuu\\
Finland}
\email{jouni.rattya@uef.fi}

\author[Siyu Wang]{Siyu Wang}
\address{School of Mathematics and Statistics\\
Northeast Normal University\\ Changchun\\ Jilin 130024\\ P.R.China}
\email{wangsy696@nenu.edu.cn}

\author{Fanglei Wu}
\address{University of Eastern Finland, P.O.Box 111, 80101 Joensuu, Finland}
\email{fangleiwu1992@gmail.com}

\maketitle

\section{Introduction and main results}

Each radial weight $\nu$ and analytic function $f$ on the unit disc $\D$ induce the small Hankel operator $h_f^\nu$ defined by
	$$
	h_f^\nu(g)(z)=P_\nu(f\overline{g})(z)=\int_\D f(\z)\overline{g(\z)B_{z}^\nu(\z)}\nu(\zeta)\,dA(\z), \quad z\in\D,
	$$
where $P_\nu$ is the orthogonal Bergman projection from the Lebesgue space $L^2_\nu$ to the Bergman space $A^2_\nu$, and $B^\nu_z$ is the reproducing kernel of $A^2_\nu$. The main goal of this study is to characterize the boundedness of $h^\nu_f:A^p_\om\rightarrow A^q_\nu$ on the range $0<p,q<\infty$ when $\omega,\nu$ belong to the class $\mathcal{D}$ of radial weights admitting certain two-sided doubling properties, and the analytic symbol $f$ satisfies
	\begin{equation}\label{fubinicondition}
  \|f\|_{A^1_{\nu_{\log}}}=\int_\D|f(z)|\left(\log\frac{e}{1-|z|}\right)\nu(z)\,dA(z)<\infty.
	\end{equation}
The novelty of our findings stems from three facts. First, we consider the full range $0<p,q<\infty$ in the two-weight setting, and thus, in particular, overcome the essential difficulties related to the case when $q<1$. The proofs require a new innovation involving a boosting of the order of the derivative in the definition of the Hankel operator. Second, we do not impose any initial hypotheses on the interrelationship between the weights $\om$ and $\nu$; they are just any prefixed weights in the class $\DDD$. This also completes earlier results in the literature and is achieved in part with the aid of the boosting. Third, the class $\DDD$ we consider is known to be vast, and in a sense the most large class of radial weights that contains the standard weights as prototypes. Therefore our study extends known results to this setting. The generalizations we obtain come by no means for free and actually require a number of technically demanding steps. As an application of our discoveries, we establish a natural weak factorization for the Bergman space $A^p_\om$ with $1<p<\infty$ and $\om\in\DDD$. This relies on the well-known connection between a description of certain bounded bilinear Hankel forms and the said factorizations.

It is well-known that the big Hankel operator and small Hankel operator are essentially the same on the Hardy spaces in the sense of unitary transformation~\cite{Nikolski} while they are quite different on the Bergman spaces~\cite{zhu}. Nonetheless, while the literature concerning the big Hankel operator is vast, only few studies concern the boundedness of the small Hankel operator on Bergman spaces even for an analytic symbol. As far as we know, Bonami and Luo~\cite{B-L} considered the boundedness of the small Hankel operator $h_{f}$ from an unweighted Bergman space $A^p$ to $A^q$ for $0<p,q<\infty$, Aleman and Constantin~\cite{A-C-2004} studied the boundedness of $h_{f}$ on vector-valued $A^2_\alpha$, and Pau and Zhao \cite{Pau-Zhao} established the weak factorization for $A^p_\alpha$ in terms of characterizing the bounded operator $h_{f}^{\alpha}:A^p_\beta\rightarrow A^q_\gamma$ on the unit ball for $1\le q<p<\infty$ and some ranges of $-1<\alpha,\beta,\gamma<\infty$. Moreover, recently small Hankel operator $h_{f}^{\alpha}$ with operator valued symbol $f$ on standard weighted vector-valued Bergman spaces were characterized by B\'{e}koll\'{e}, Defo, Tchoundja, and Wick \cite{B-O-T-W} and Oliver \cite{O}, and a description for the boundedness of $h_{f}^{\omega}:A^p_\om\to A^q_\om$ when $1<q<p<\infty$ and $\omega\in\widehat{\mathcal{D}}$ was established by Korhonen and R\"{a}tty\"{a} \cite{KR}.

The problem of characterizing bounded operators $h^\nu_f:A^p_\om\rightarrow A^q_\nu$ is analogous to the well-known problem of characterizing the boundedness of singular integral operators acting between different Lebesgue spaces. These kind of two-weight problems have attracted a considerable amount of attention in complex and harmonic analysis, and are closely connected to other intriguing questions on the area \cite{Hy2012,Hy-1,Hy-2,LSCU2014}. Such a problem appearing in operator theory can be attributed to the series of studies of boundedness of linear operators from a weighted Lebesgue space to another. One of the most attractive problems over the past several decades is to characterize when the Bergman projection $P_{\omega}:L^{p}_{\nu}\to A_{\nu}^{p}$ is bounded. The most commonly known result on the Bergman projection is due to Bekoll\'e and Bonami \cite{B,BB1978}, and concerns the case when $\nu$ is an arbitrary weight and $\om$ is standard; when both $\omega,\nu$ are certain radial weights, see \cite{Do,PR2014,PelRat2020} for recent results and the reference therein. Most of the known results were obtained for $p\in(1,\infty)$, while in the endpoint case $p=1$, the weak (1,1) property of Bergman projection were demonstrated in \cite{B,D-H-Z-Z}. Similar studies for other classical operators intimately related to the Bergman projection on weighted Bergman spaces can be found in \cite{DGWW,Nikolski,Pau-Zhao-,PRS1} for Toeplitz operators, \cite{Hu,Pau-Zhao-Zhu,PPR2020,Peller} for Hankel operators and \cite{Zhao-Zhou} for Forelli-Rudin type operators.

We now proceed towards the exact statements via necessary definitions. Let $\H(\D)$ denote the space of analytic functions on the unit disc $\D=\{z\in\C:|z|<1\}$. A non-negative function $\om\in L^1(\D)$ such that $\om(z)=\om(|z|)$ for all $z\in\D$ is called a radial weight. For $0<p<\infty$ and such an $\omega$, the Lebesgue space $L^p_\om$ consists of measurable complex-valued functions $f$ on $\D$ such that
    $$
    \|f\|_{L^p_\omega}^p=\int_\D|f(z)|^p\omega(z)\,dA(z)<\infty,
    $$
where $dA(z)=\frac{dx\,dy}{\pi}$ is the normalized area measure on $\D$. The corresponding weighted Bergman space is $A^p_\om=L^p_\omega\cap\H(\D)$. As usual, we write $A^p_\alpha$ for the classical weighted Bergman spaces induced by the standard radial weight $\om(z)=(\a+1)(1-|z|^2)^\alpha$ with $-1<\alpha<\infty$. Throughout this paper we assume $\widehat{\om}(z)=\int_{|z|}^1\om(s)\,ds>0$ for all $z\in\D$, for otherwise $A^p_\om=\H(\D)$.

For a radial weight $\nu$, the norm convergence in $A^2_\nu$ implies the uniform convergence on compact subsets, and therefore the Hilbert space $A^2_\nu$ is a closed subspace of $L^2_\nu$ and the orthogonal Bergman projection $P_\nu:L^2_\nu\to A^2_\nu$ is given by
	$$	
	P_\nu(g)(z)=\int_\D g(z)\overline{B^\nu_z(\z)}\nu(\z)\,dA(\z),\quad z\in\D,
	$$
where $B^\nu_z$ is the reproducing kernel of $A^2_\nu$, associated to the point $z\in\D$. As usual, we write $B^\alpha_z$ for the kernel of $A^2_\alpha$, and $P_\alpha$ stands for the corresponding Bergman projection.

A radial weight $\om$ belongs to the class~$\DD$ if there exists a constant $C=C(\om)\ge1$ such that $\widehat{\om}(r)\le C\widehat{\om}(\frac{1+r}{2})$ for all $0\le r<1$. Moreover, if there exist $K=K(\om)>1$ and $C=C(\om)>1$ such that $\widehat{\om}(r)\ge C\widehat{\om}\left(1-\frac{1-r}{K}\right)$ for all $0\le r<1$, then we write $\om\in\Dd$. In other words, $\om\in\Dd$ if there exist $K=K(\om)>1$ and $C'=C'(\om)>0$ such that
	\begin{equation}\label{Dcheck}
	   \widehat{\om}(r)\le C'\int_r^{1-\frac{1-r}{K}}\om(t)\,dt,\quad 0\le r<1.
	\end{equation}
The intersection $\DD\cap\Dd$ is denoted by $\DDD$. The class $\RR\subset\DDD$ of regular weights consists of those radial weights for which $\widehat{\om}(r)\asymp\omega(r)(1-r)$ for all $0\le r<1$. It is known that $\DD$ and $\DDD$ and arise naturally in many instances in the operator theory of Bergman spaces induced by radial weights~\cite{PelRat2020}. Apart from this we refer to \cite{PelSum14,PR} and references therein for basic properties and more on different classes of radial weights.

Throughout the paper $\frac1p+\frac{1}{p'}=1$ for $1<p<\infty$. Further, the letter $C=C(\cdot)$ will denote an absolute constant whose value depends on the parameters indicated in the parenthesis, and which may change from one occurrence to another. If there exists a constant
$C=C(\cdot)>0$ such that $a\le Cb$, then it is written either $a\lesssim b$ or $b\gtrsim a$. In particular, if $a\lesssim b$ and
$a\gtrsim b$, then it is denoted by $a\asymp b$ and said that $a$ and $b$ are comparable.

The main findings of this study are stated in this section while the proofs are given in the forthcoming ones in the numerical order. Our first result generalizes the two-weight case of \cite[Theorem~3]{Pau-Zhao} in the unit disc setting in two ways. Namely, we work without any initial hypotheses on the interrelationship between the weights involved, and thus cover also the cases excluded in the said theorem, and we consider doubling weights instead of the standard radial weights.

\begin{theorem}\label{theorem}
Let $1<q<p<\infty$ and $\om,\nu\in\DDD$, and let $f\in A^1_{\nu_{\log}}$. Then there exists an $N=N(\om,\nu,p,q)\in\N$ such that $h^\nu_f: A^p_{\om}\rightarrow A^q_{\nu}$ is bounded if and only if
	\begin{equation}\label{Eq:theorem}
	\int_\D\left|f^{(n)}(z)\right|^{\frac{pq}{p-q}}(1-|z|)^{n\frac{pq}{p-q}}\left(\frac{\widehat{\nu}(z)}{\whw(z)}\right)^{\frac{p}{p-q}}\frac{\widehat{\om}(z)}{1-|z|}\,dA(z)<\infty
	\end{equation}
for some (equivalently for all) $n\ge N$. Moreover,
	$$
	\|h^\nu_f\|^{\frac{pq}{p-q}}_{A^p_\om\rightarrow A^{q}_{\nu}}
	\asymp
	\int_\D\left|f^{(n)}(z)\right|^{\frac{pq}{p-q}}(1-|z|)^{n\frac{pq}{p-q}}\left(\frac{\widehat{\nu}(z)}{\whw(z)}\right)^{\frac{p}{p-q}}\frac{\widehat{\om}(z)}{1-|z|}\,dA(z)+\sum_{j=0}^{n-1}|f^{(j)}(0)|^{\frac{pq}{p-q}}
	$$
for each fixed $n\ge N$.
\end{theorem}

We make several observations on the upper triangular case $1<q<p<\infty$ considered in Theorem~\ref{theorem}. First of all, the appearance of the higher order derivative, or alternatively some fractional derivative, in the statement is inevitable if no initial hypotheses on the interrelationship between $\omega$ and $\nu$ is imposed. Namely, if $\om,\nu\in\DDD$ then the function $\eta=\left(\frac{\widehat{\nu}}{\whw}\right)^{\frac{p}{p-q}}\frac{\widehat{\om}}{1-|\cdot|}$ is not necessarily integrable over $\D$, yet $h^\nu_f: A^p_{\om}\rightarrow A^q_{\nu}$ may certainly be bounded without forcing $f^{(n)}$ to vanish identically. Further, since $\om,\nu\in\DDD$ we may choose $n=n(\om,\nu,p,q)\in\N$ such that the weight appearing in \eqref{Eq:theorem} becomes integrable and the condition makes perfect sense. If $\eta$ happens to be a weight in $\DDD$, then we may certainly choose $n=0$ because the Littlewood-Paley estimate \cite[Theorem~5]{PelRat2020} shows that the finiteness of the quantity \eqref{Eq:theorem} is independent of $n$.

Even if the proof of \cite[Theorem~3]{Pau-Zhao} works as a guideline for us to obtain Theorem~\ref{theorem}, there are a number of things that we have to do differently. The boosting of the order of the derivative is one of the crucial points in the proof of Theorem~\ref{theorem}. In fact, it is involved in the proofs of all main results of this study. In Section~\ref{Sec:boosting} we give the details of this argument and show that, for each $n\in\N$ and for each radial weight~$\nu$, there exists a radial weight~$\mu^n_\nu$ such that
	$$
	\int_\D f(\zeta)\overline{B_z^{\mu^n_\nu}(\zeta)}\nu(\zeta)\,dA(\zeta)=\frac{d^n}{dz^n}\left(z^nf(z)\right),\quad f\in A^1_\nu.
	$$
This integral formula yields the novel representation
	$$
	h_f^\nu(g)(z)
	=\int_{\D}\overline{g(\z)B_z^{\nu}(\z)}\left(\frac{d^n}{d\z^n}\left(\z^n f(\z)\right)\right)\mu^n_{\nu}(\z)\,dA(\z),\quad g\in H^\infty,
	$$
which is the creature that we will mainly work with. Since we do not require any extra hypotheses in Theorem~\ref{theorem}, one immediately sees the advantage of this representation in terms of the boosting through comparing the special case $\alpha=\alpha_2$ of Pau and Zhao~\cite[Theorem 3]{Pau-Zhao} to our theorem. It tells us that, at least in the unit disc case when $\alpha=\alpha_2$, the extra restriction appearing in the statement of \cite[Theorem 3]{Pau-Zhao} can be dropped.

Apart from the boosting and several technical results on doubling weights, the proof of the sufficiency in Theorem~\ref{theorem} relies on norm estimates of the Bergman reproducing kernel. We will use the estimate
	\begin{equation}\label{kernel}
	\left\|\left(B^\nu_z\right)^{(k)}\right\|_{A^p_\om}^p\lesssim\int_0^{|z|}\frac{\widehat{\om}(t)}{\widehat{\nu}(t)^p(1-t)^{p(k+1)}}\,dt+1,\quad z\in\D,
	\end{equation}
valid for $0<p<\infty$, $\nu\in\DD$, $k\in\N\cup\{0\}$ and all radial weights $\om$, repeatedly in our arguments. The proof of \eqref{kernel} and more can be found in \cite{PR2014}.

The proof of the necessity in Theorem~\ref{theorem} is technically demanding and relies on the atomic decomposition of functions in $A^p_\om$~\cite[Theorems~1 and 2]{PRS-2021} and a modification of a well-known argument involving the Rademacher functions and Khinchine's inequality. It also requires a good understanding of the boundedness of the Bergman projection
	$$
	P_\nu(f)(z)=\int_\D f(\zeta)\overline{B^\nu_z(\zeta)}\nu(\z)\,dA(\zeta)
	$$
acting on $L^p_\om$, and duality relations of Bergman spaces with respect to different pairings, two important concepts that are certainly interrelated. The references \cite{PR2014} and \cite{PelRat2020} serve us with regard to these matters.

Pau and Zhao~\cite{Pau-Zhao} successfully used the small Hankel operators induced by analytic symbols, or more precisely bilinear Hankel forms, to obtain a weak factorization of $A^p_\alpha$ in the case $1<p<\infty$ on the unit ball of $\C^n$, and thus efficiently completed the existing literature. Following them, the work done for Theorem~\ref{theorem} allows us to do the same in our setting. To give the precise statement, some more notation is needed. Given two Banach spaces $A,B$ of functions, which are defined on the same domain, the weakly factored space $A\odot B$ is defined as the completion of finite sums
	\[
	f=\sum_{k}\varphi_{k}\psi_{k},\quad\{\varphi_{k}\}\subset A,\quad~\{\psi_{k}\}\subset B,
	\]
with the norm
\[
\|f\|_{A\odot B}=\inf_{\{\varphi_{k},\psi_{k}:
f=\sum_{k}\varphi_{k}\psi_{k},\varphi_{k}\in A,\psi_{k}\in B\}}\sum_{k}\|\varphi_{k}\|_{A}\|\psi_{k}\|_{B}
<\infty.
\]

By adopting the notation $\widetilde{\tau}(r)=\widehat{\tau}(r)/(1-r)$ for each radial weight $\tau$ and all $0\le r<1$, the weak factorization that we obtain for weighted Bergman spaces generalizes \cite[Theorem~1]{Pau-Zhao} in the unit disc setting to doubling weights and reads as follows.

\begin{theorem}\label{weak factorization-t1}
Let $1<p_{1},p_{2},q<\infty$ and $\omega,\nu,\eta\in\mathcal{D}$ such that $q^{-1}=p_{1}^{-1}+p_{2}^{-1}$
and
	\begin{equation}\label{weak factorization-eq3}
	\widetilde{\eta}^{\frac{1}{q}}\asymp\widetilde{\omega}^{\frac{1}{p_{1}}}\widetilde{\nu}^{\frac{1}{p_{2}}}.
	\end{equation}
Then
	\[
	A_{\eta}^{q}=A_{\omega}^{p_{1}}\odot A_{\nu}^{p_{2}},
	\]
that is, $A_{\eta}^{q}$ and $A_{\omega}^{p_{1}}\odot A_{\nu}^{p_{2}}$ are the same as function spaces and their norms are equivalent.
\end{theorem}

One may very well criticize the above weak factorization because a strong factorization analogous to this exists when $\om=\nu=\eta\in\DD$ by \cite[Theorem~5]{KR}. It states that, if $0<p,p_1,p_2<\infty$ and $\om\in\DD$ such that $p^{-1}=p_1^{-1}+p_2^{-1}$, then for each $f\in A^p_\om$ there exist $f_1\in A^{p_1}_\om$ and $f_2\in A^{p_2}_\om$ such that $f=f_1f_2$ and
    \begin{equation}\label{lkjhlajhsgf}
    \|f_1\|_{A^{p_1}_\om}^{p}\|f_2\|_{A^{p_2}_\om}^{p}
    \le\frac{p}{p_1}\|f_1\|_{A^{p_1}_\om}^{p_1}+\frac{p}{p_2}\|f_2\|_{A^{p_2}_\om}^{p_2}
    \le C\|f\|_{A^p_\om}^p\le C\|f_1\|_{A^{p_1}_\om}^{p}\|f_2\|_{A^{p_2}_\om}^{p}
    \end{equation}
for some constant $C=C(\om)>0$. The method of proof employed in \cite{KR} is inherited from Horowitz' original probabilistic argument~\cite{Horowitz} which, apart from being clever, seems to be hardly applicable in the case when the weights are distinct. At this point we do not even know whether such strong factorization should exist in our setting, and thus this matter remains unsettled here.

It is justified to ask when the asymptotic equality \eqref{weak factorization-eq3} can be satisfied. The case of standard weights with equality is discussed in \cite{Pau-Zhao}, but in our setting the question does not have an immediate answer. However, we can say that if $1<q<\infty$ and $\eta,\om\in\DDD$ are fixed, then there exist $1<p_2<p_1<\infty$ and $\nu\in\DDD$ such that \eqref{weak factorization-eq3} is valid. Namely, by using \cite[Lemma~2.1]{PelSum14} and \cite[Lemma~B]{PRS1} one can see that if the quotient $p_1/p_2$ is sufficiently large, depending on $\eta$ and $\om$, then $\nu$ defined by the identity $\widetilde{\eta}^{\frac{1}{q}}=\widetilde{\omega}^{\frac{1}{p_{1}}}\widetilde{\nu}^{\frac{1}{p_{2}}}$ satisfies $\widetilde{\nu}\in\DDD$, but this is equivalent to $\nu\in\DDD$ by \cite[Theorem~9]{PR2020}.

We now return back to the small Hankel operator. By Theorem~\ref{theorem}, in the upper triangular case $1<q<p<\infty$, the condition that characterizes the boundedness of $h^\nu_f: A^p_{\om}\rightarrow A^q_{\nu}$ is an integral condition which in certain cases reduces to the requirement that $f^{(n)}$ belongs to a specific Bergman space. Our next result shows that in the case $p\le q$ and $q>1$ it is a condition on the growth of the maximum modulus of $f^{(n)}$ that characterizes the boundedness.


\begin{theorem}\label{theorem1}
Let $0<p\le q<\infty$, $1<q<\infty$ and $\om,\nu\in\DDD$, and let $f\in A^1_{\nu_{\log}}$. Then there exists an $N=N(\om,\nu,p,q)\in\N$ such that $h^\nu_f: A^p_{\om}\rightarrow A^q_{\nu}$ is bounded if and only if
  \begin{equation}\label{Eq:theorem2}
  \sup_{z\in\D}\left|f^{(n)}(z)\right|(1-|z|)^n
	\frac{\left(\widehat{\nu}(z)(1-|z|)\right)^{\frac{1}{q}}}{\left(\whw(z)(1-|z|)\right)^{\frac{1}{p}}}<\infty
  \end{equation}
for some (equivalently for all) $n\ge N$. Moreover,
  $$
  \|h^\nu_f\|_{A^p_\om\rightarrow A^{q}_{\nu}}
	\asymp\sup_{z\in\D}\left|f^{(n)}(z)\right|(1-|z|)^n
	\frac{\left(\widehat{\nu}(z)(1-|z|)\right)^{\frac{1}{q}}}{\left(\whw(z)(1-|z|)\right)^{\frac{1}{p}}}+\sum_{j=0}^{n-1}|f^{(j)}(0)|
  $$
for each fixed $n\ge N$.
\end{theorem}

The special case $q=p\in(1,\infty)$ and $\nu=\om\in\DDD$ in Theorem~\ref{theorem1} shows that $h^\om_f: A^p_{\om}\rightarrow A^p_{\om}$ is bounded if and only if $f$ belongs to the Bloch space $\B$ by \cite[Proposition~8]{zhu2}. Moreover, if $f\in\B$ and $A^p_\om\subset A^q_\nu$, then \eqref{Eq:theorem2} is satisfied by \cite[Theorem~1]{PR2015}. Recall that the Bloch space consists of $f\in\H(\D)$ such that
	$$
	\|f\|_{\B}=\sup_{z\in\D}|f'(z)|(1-|z|^2)+|f(0)|<\infty.
	$$

The proof of the sufficiency in Theorem~\ref{theorem1} is pretty straightforward now that useful estimates and auxiliary results are already established in the proof of Theorem~\ref{theorem}. In addition to standard arguments the only extra tool needed is the recent characterization of the $q$-Carleson measures for $A^p_\om$ with $\om\in\DDD$ in terms of quantities involving pseudohyperbolic discs~\cite[Theorem~2]{LiuRattya}. Recall that a positive Borel measure $\mu$ on $\D$ is a $q$-Carleson measure for $A^p_\om$ if the identity operator $Id: A^p_\om\rightarrow L^q_\mu$ is bounded.

For the necessity we will use a kernel estimate which slightly differs from \eqref{kernel} and which can be found in the recent literature~\cite{PR2017}. It states that, for each radial weight $\om$, we have
	\begin{equation*}
  \int_{\D}|(1-\overline{z}\z)^k B^\nu_z(\z)|^p\om(\z)\,dA(\z)
	\lesssim\int_0^{|z|}\frac{\whw(t)}{\widehat{\nu}(t)^p(1-t)^{p(1-k)}}\,dt+1,\quad z\in\D,
	\end{equation*}
provided $2\le p<\infty$, $k\in\N\cup\{0\}$ and $\nu\in\DD$. This estimate also plays an important role in the proof of the forthcoming result concerning the case $0<p\le 1=q$.

It is worth mentioning that an alternative characterization of the boundedness of $h^\nu_f$ in the case of Theorem~\ref{theorem1} can be obtained also in terms of so-called general fractional derivatives which would then play a role similar to the boosting in the proof. In \cite{Per2020,zhu1994}, the general fractional derivative of an analytic function $f(z)=\sum_{n=0}^\infty \widehat{f}(n)z^n$ related to two radial weights $\om$ and $\nu$ is defined by
	$$
	R^{\om,\nu}(f)(z)=\int_\D f(\z)\overline{B^\nu_z(\z)}\om(\z)\,dA(z),\quad z\in\D.
	$$
By imitating the proof of Theorem~\ref{theorem1}, one may easily show that there exists an $\alpha_0=\alpha_0(\om,\nu,p,q)$ such that $h^\nu_f: A^p_{\om}\rightarrow A^q_{\nu}$ is bounded under the hypotheses of the theorem if and only if
	$$
	\sup_{z\in\D}\left|R^{\nu,\alpha}(f)(z)\right|
	\frac{(1-|z|)^{\alpha+2-\frac{1}{p}-\frac{1}{q'}}}{\whw(z)^{\frac{1}{p}}\widehat{\nu}(z)^{\frac{1}{q'}}}<\infty
	$$
for some (equivalently for all) $\alpha>\alpha_0$, and moreover,	
	$$
 \|h^\nu_f\|_{A^p_\om\rightarrow A^{q}_{\nu}}
	\asymp\sup_{z\in\D}\left|R^{\nu,\alpha}(f)(z)\right|
	\frac{(1-|z|)^{\alpha+2-\frac{1}{p}-\frac{1}{q'}}}{\whw(z)^{\frac{1}{p}}\widehat{\nu}(z)^{\frac{1}{q'}}}
	$$
for each fixed $\alpha>\alpha_0$. This asymptotic equality together with Theorem \ref{theorem1} immediately yields a characterization of certain Bloch-type space. In particular, if $q=p\in(1,\infty)$ and $\nu=\om\in\DDD$, then we obtain that
	$$
	\|f\|_\B\asymp\sup_{z\in\D}\left|R^{\om,\alpha}(f)(z)\right|\frac{(1-|z|)^{\alpha+1}}{\whw(z)}
	$$
for each sufficiently large $\alpha$ depending on $\om$. Interested readers should compare this novel characterization of the Bloch space, proved with the aid of the small Hankel operator, with another fractional derivative characterization of $\B$ recently discovered in \cite{PR2017}.

The next two results concern the case $q=1$. We begin with the range $1<p<\infty$, where the statement involves the $p$-Bloch Carleson measures which are those positive Borel measures $\mu$ on $\D$ for which the identity operator $Id:\B\rightarrow L^p_\mu$ is bounded.


\begin{theorem}\label{theorem2}
Let $1<p<\infty$ and $\om,\nu\in\DDD$, and let $f\in A^1_{\nu_{\log}}$. Then there exists an $N=N(\om,\nu,p)\in\N$ such that $h^\nu_f: A^p_{\om}\rightarrow A^1_{\nu}$ is bounded if and only if
	\begin{equation}\label{Eq:theorem3}
	|f^{(n)}(z)|^{p'}(1-|z|^{2})^{np'}\left(\frac{\widehat{\nu}(z)}{\widehat{\omega}(z)}\right)^{p'}\frac{\widehat{\om}(z)}{1-|z|}\,dA(z)
	\end{equation}
is a $p'$-Bloch Carleson measure for some (equivalently for all) $n\geq N$. Moreover,	
	$$
	\|h_{f}^{\nu}\|_{A^{p}_{\omega}\to A^{1}_{\nu}}^{p'}
	\asymp\sup_{h\in\B\setminus\{0\}}\frac{\int_{\mathbb{D}}\left|h(z)\right|^{p'}\left|f^{(n)}(z)\right|^{p'}(1-|z|^{2})^{np'}
	\left(\frac{\widehat{\nu}(z)}{\widehat{\omega}(z)}\right)^{p'}\frac{\widehat{\om}(z)}{1-|z|}\,dA(z)}{\|h\|_{\mathcal{B}}^{p'}}
	+\sum_{j=0}^{n-1}|f^{(j)}(0)|^{p'}
	$$
for each fixed $n\ge N$.
\end{theorem}

Theorem~\ref{theorem2} has an obvious defect. Namely, the existing literature does not seem to offer a reasonable characterization of the $p$-Bloch Carleson measures, yet some attempts can be found in \cite{GPP-GR2008}. Therefore a critical reader may very well say that Theorem~\ref{theorem2} basically offers a possibly non-trivial reformulation of the problem of characterizing the boundedness of $h^\nu_f: A^p_{\om}\rightarrow A^1_{\nu}$ in terms of the $p$-Bloch Carleson measures. However, Theorem~\ref{theorem2} implies that there exists an $N=N(p,\om,\nu)\in\N$ such that for each fixed $n\ge N$ we have
	$$
	\|h^\nu_{f}\|_{A^p_\om\rightarrow A^1_\nu}^{p'}
	\lesssim\int_\D\left|f^{(n)}(z)\right|^{p'}(1-|z|)^{np'}
		\left(\frac{\widehat{\nu}(z)}{\whw(z)}\right)^{p'}\left(\log\frac{e}{1-|z|}\right)^{p'}\frac{\widehat{\om}(z)}{1-|z|}\,dA(z)+\sum_{j=0}^{n-1}|f^{(j)}(0)|^{p'}.
	$$
We have been unable to judge whether or not $\lesssim$ above can be replaced by $\asymp$. This case remains unsettled and seems to call for further research.

Most of the tools needed to achieve Theorem~\ref{theorem2} are already present in the proofs of Theorems~\ref{theorem} and~\ref{theorem1}. The only extra ingredient needed is the duality relation $(A^1_{\nu})^\star\simeq\B$ via the $A^2_\nu$-pairing, valid for $\nu\in\DDD$ by \cite[Theorem~3]{PelRat2020}.

The next theorem concerns the case $0<p\le1=q$.


\begin{theorem}\label{theorem3}
Let $0<p\le1$ and $\om,\nu\in\DDD$, and let $f\in A^1_{\nu_{\log}}$. Then there exists an $N=N(\om,\nu,p)\in\N$ such that $h^\nu_f: A^p_{\om}\rightarrow A^1_{\nu}$ is bounded if and only if
	\begin{equation}\label{Eq:Thm3-hypothesis}
  \sup_{z\in\D}\left|f^{(n)}(z)\right|(1-|z|)^n
	\frac{\widehat{\nu}(z)(1-|z|)}{\left(\whw(z)(1-|z|)\right)^{\frac{1}{p}}}\log\frac{e}{1-|z|}<\infty
  \end{equation}
for some (equivalently for all) $n\ge N$. Moreover,
	$$
	\|h^\nu_{f}\|_{A^p_\om\rightarrow A^1_\nu}
	\asymp\sup_{z\in\D}\left|f^{(n)}(z)\right|(1-|z|)^n
	\frac{\widehat{\nu}(z)(1-|z|)}{\left(\whw(z)(1-|z|)\right)^{\frac{1}{p}}}\log\frac{e}{1-|z|}+\sum_{j=0}^{n-1}|f^{(j)}(0)|
	$$
for each fixed $n\ge N$.
\end{theorem}

The case $p=1$ and $\nu=\om\in\DDD$ in Theorem~\ref{theorem3} shows that $h^\om_f: A^1_{\om}\rightarrow A^1_{\om}$ is bounded if and only if
	\begin{equation*}
  \sup_{z\in\D}\left|f^{(n)}(z)\right|(1-|z|)^n\log\frac{e}{1-|z|}<\infty.
  \end{equation*}
It is easy to see that this condition is equivalent to
	\begin{equation}\label{msufk}
  \sup_{z\in\D}\left|f'(z)\right|(1-|z|)\log\frac{e}{1-|z|}<\infty.
  \end{equation}
Moreover, if $f$ satisfies \eqref{msufk} and $A^p_\om\subset A^1_\nu$, then \eqref{Eq:Thm3-hypothesis} is satisfied by \cite[Theorem~1]{PR2015}.

The last two theorems concern the case $0<q<1$. The existing results concerning the boundedness of the small Hankel operator acting between the Bergman spaces induced by standard weights when the parameter $q$ of the target space is smaller than 1 can be found in \cite{B-O-T-W,B-L}. The arguments employed there include weak Bergman spaces and the weak type (1,1) property of the Bergman projection. Observe that there is a close and well-known connection between the weak (1,1)- and $A_{\omega}^{q}$-norms~\cite{G-2014}. These techniques are not applicable in the setting of weights in $\mathcal{D}$  because of the lack of a suitable description of the weak type (1,1) boundedness of the Bergman projection. Therefore we argue in a different way and establish a more direct proof which also applies in the classical case.

We begin with the range $1<p<\infty$.


\begin{theorem}\label{theorem4}
Let $0<q<1$, $1<p<\infty$ and $\om,\nu\in\DDD$, and let $f\in A^1_{\nu_{\log}}$. Then there exists an $N=N(\om,\nu,p,q)\in\N$ such that $h^\nu_f: A^p_{\om}\rightarrow A^q_{\nu}$ is bounded if and only if
		\begin{equation}\label{result4}
    \int_\D\left|f^{(n)}(z)\right|^{p'}(1-|z|)^{np'}
		\left(\frac{\widehat{\nu}(z)}{\whw(z)}\right)^{p'}\frac{\widehat{\om}(z)}{1-|z|}\,dA(z)<\infty
		\end{equation}
for some (equivalently for all) $n\ge N$. Moreover,
	$$
	\|h^\nu_{f}\|^{p'}_{A^p_\om\rightarrow A^q_\nu}
	\asymp\int_\D\left|f^{(n)}(z)\right|^{p'}(1-|z|)^{np'}
	\left(\frac{\widehat{\nu}(z)}{\whw(z)}\right)^{p'}\frac{\widehat{\om}(z)}{1-|z|}\,dA(z)+\sum_{j=0}^{n-1}|f^{(j)}(0)|^{p'}.
	$$
for each fixed $n\ge N$.
\end{theorem}

In the case $\nu=\om\in\DDD$ we deduce that
	$$
	\|h^\nu_{f}\|^{p'}_{A^p_\om\rightarrow A^q_\nu}
	\asymp\int_\D\left|f^{(n)}(z)\right|^{p'}(1-|z|)^{np'}
	\frac{\widehat{\om}(z)}{1-|z|}\,dA(z)+\sum_{j=0}^{n-1}|f^{(j)}(0)|^{p'}.
	$$
The right hand side is comparable to $\|f\|_{A^{p'}_\om}^{p'}$ by \cite[Theorem~5]{PelRat2020} and \cite[Proposition~5]{PRS1}.

An intriguing but not that obvious observation here is that \eqref{result4} actually also characterizes a certain three-weight inequality for small Hankel operator. To make this precise, following \cite{PeraRa}, set
	$$
	V_{\nu,q}(z)
	=\left(\frac{1}{q}-1\right)\widehat{\nu}(z)^{\frac{1}{q}}(1-|z|)^{\frac{1}{q}-2}
	+\frac{\nu(z)}{q} \widehat{\nu}(z)^{\frac{1}{q}-1}(1-|z|)^{\frac{1}{q}-1}, \quad z \in \mathbf{\D},
	$$
for each $\nu$ and $q$. Then the method employed in the proof of Theorem~\ref{theorem4} shows that $h^\nu_f: A^p_{\om}\rightarrow A^1_{V_{\nu,q}}$ is bounded if and only if \eqref{result4} is satisfied, under the hypotheses of the theorem. This makes us realize that $A^q_\nu$ is not the best target space for \eqref{result4} because $A^q_\nu\subsetneq A^1_{V_{\nu,q}}$. Details are left for an interested reader.

Our last main result concerns the case $0<p\le1$.


\begin{theorem}\label{theorem5}
Let $0<q<1$, $0<p\leq1$ and $\om,\nu\in\DDD$, and let $f\in A^1_{\nu_{\log}}$. Then there exists an $N=N(\om,\nu,p,q)\in\N$ such that $h^\nu_f: A^p_{\om}\rightarrow A^q_{\nu}$ is bounded if and only if
	\begin{equation}\label{result5}
  \sup_{z\in\D}\left|f^{(n)}(z)\right|(1-|z|)^n
	\frac{\widehat{\nu}(z)(1-|z|)}{\left(\whw(z)(1-|z|)\right)^{\frac{1}{p}}}<\infty
  \end{equation}
for some (equivalently for all) $n\ge N$. Moreover,
	$$
	\|h^\nu_{f}\|_{A^p_\om\rightarrow A^q_\nu}
	\asymp\sup_{z\in\D}\left|f^{(n)}(z)\right|(1-|z|)^n
	\frac{\widehat{\nu}(z)(1-|z|)}{\left(\whw(z)(1-|z|)\right)^{\frac{1}{p}}}+\sum_{j=0}^{n-1}|f^{(j)}(0)|
	$$
for each fixed $n\ge N$.
\end{theorem}

If $f\in\B$ and $A^p_\om\subset A^1_\nu$, then \eqref{result5} is satisfied by \cite[Theorem~1]{PR2015}. As in the case of Theorem~\ref{theorem4}, the condition \eqref{result5} in Theorem~\ref{theorem5} also characterizes the boundedness of the small Hankel operator in a certain three-weight case, the method of proof being the same. To be precise, under the hypotheses of Theorem~\ref{theorem5}, $h^\nu_f: A^1_{V_{\om,p}}\rightarrow A^1_{V_{\nu,q}}$ is bounded if and only if \eqref{result5} holds.

We finish the introduction by discussing briefly some natural open problems and directions for further research. Probably the most natural next question to study would be to consider the compactness of the small Hankel operator as well as the schatten class membership for the Hilbert space case, see \cite{J1988,JPR1987,zhu} and reference therein for the standard weight case with this regard. Certainly, all these questions can be considered in higher dimensions as well, in which case the reader should consult \cite{B-O-T-W,O,Pau-Zhao} for earlier studies. In addition, in this paper the symbol inducing the small Hankel operator is confined to be analytic, while the operator itself is certainly well-defined for a more general symbol. However, the method to study those operators should be different from the method we used here, and this certainly makes the problem more challenging, see \cite[Theorem~2]{PR2017}. Further, in the case of analytic symbols, the method we employed does not seem to work in certain cases if the weights involved are only assumed to belong to $\DD$. For example, because of the lack of the atomic decomposition for functions in $A^p_\om$ when $\om\in \DD$, the proof of necessity of Theorem \ref{theorem} falls apart all together.

The interest in studying the three-weight case of the small Hankel operator with weights in~$\DD$ also stems from the fact that it might possibly provide a tool to tackle a certain open question concerning integration operators. To be more concrete, a suitable weak factorization of $A^p_\om$ might be obtained in the same fashion as in \cite{Pau-Zhao} if we can depict the bounded small Hankel operator $h^\eta_f: A^p_\om\to A^q_\nu$ provided $f\in\H(\D)$ and $\om,\nu,\eta\in\DD$. Recall now that integration operator $T_g$ induced by a $g\in\H(\D)$ is defined as
	$$
	T_g(f)(z)=\int_0^z f(\z)g'(\z)\,d\z,\quad f\in\H(\D).
	$$
The open problem we are referring to here to is to characterize the boundedness of $T_g: A^p_\om\to A^q_\nu$ for $0<q<p<\infty$ and $\om,\nu\in\DD$. This question was partially answered in \cite{PelSum14} by using so-called strong factorization of $A^p_\om$. The reason why the strong factorization used in the said paper does not seem to lead to a solution of the problem is that the two factors in the factorization depend on the same weight, see \eqref{lkjhlajhsgf}.

\section{Boosting the order of the derivative}\label{Sec:boosting}

We begin with an integral representation of the function $z\mapsto\frac{d^n}{dz^n}(z^nf(z))$ which plays a crucial role in the proofs. It allows us to boost the order of the derivative of the analytic symbol $f$ appearing in the small Hankel operator $h^\nu_f$. This in turn guarantees the integrability of certain functions induced by $\omega$ and $\nu$, and therefore allows us to avoid unnecessary hypotheses on them in the statements. As usual, since the odd moments are explicit in the Maclaurin representation of the kernel function, we use the notation $\om_x=\int_0^1r^x\om(r)\,dr$ for each radial weight $\omega$ and all $x\ge0$. We also define for each $0<r<1$ the dilation $f_r$ by $f_r(z)=f(rz)$ for all $z\in\D$.

\begin{lemma}\label{nthreproducing}
For a radial weight $\nu$ and $n\in\N$, define
	\begin{equation}\label{jhgf}
	\begin{split}
	\mu^n_\nu(z)
	&=2^n\int_{|z|}^1r_1\left(\int_{r_1}^1r_2\cdots\int_{r_{n-2}}^1r_{n-1}\left(\int_{r_{n-1}}^1\frac{\nu(r)}{r^{2n-1}}\,dr\right)dr_{n-1}\cdots\,dr_2\right)dr_1\\
	&=\sum_{j=1}^n C_{j,n} |z|^{2j-2}\int_{|z|}^1\frac{\nu(t)}{t^{2j-1}}\,dt,\quad z\in\D,
	\end{split}
	\end{equation}
where the real constants $C_{j,n}$ are uniquely determined by the identity
	$$
	\sum_{j=1}^n\frac{C_{j,n}}{k+j}=\frac{2}{(k+1)\cdots(k+n)},\quad j=1,\ldots,n,\quad k\in\N.
	$$
Then
	\begin{equation}\label{id:generalizedreproducing}
	\int_\D f(\z)\overline{B_{z}^{\mu^n_\nu}(\z)}\nu(\z)\,dA(\z)=\frac{d^n}{dz^n}(z^nf(z)),\quad f\in A^1_\nu.
	\end{equation}
Moreover, if $\nu\in\DD$, then
	\begin{equation}\label{Ieq:newweight1}
  \mu^n_\nu(z)\asymp\widehat{\nu}(z)(1-|z|)^{n-1},\quad z\in\D,
	\end{equation}
and
	\begin{equation}\label{Ieq:newweight}
  \widehat{\mu^n_\nu}(z)\asymp\widehat{\nu}(z)(1-|z|)^{n},\quad z\in\D.
	\end{equation}
\end{lemma}

\begin{proof}
First observe that the second identity in \eqref{jhgf} is verified by a routine calculation. Next, let $n\in\N$, and consider $f\in\H(\D)$ with the Maclaurin series $f(z)=\sum_{k=0}^\infty \widehat{f}(k)z^k$. Then $\frac{d^n}{dz^n}\left(z^n f(z)\right)=\sum_{k=0}^\infty(k+1)\cdots(k+n)\widehat{f}(k)z^k$ for all $z\in\D$. Further, a direct calculation shows that $\mu^n_\nu$ satisfies
	\begin{equation}\label{id: mu}
  (\mu^n_\nu)_{2k+1}=\frac{\nu_{2k+1}}{(k+1)\cdots(k+n)},\quad k\in\N\cup\{0\}.
	\end{equation}
Therefore Fubini's theorem yields
	\begin{align*}
  \int_\D f_r(\z)\overline{B_{z}^{\mu^n_\nu}(\z)}\nu(\z)\,dA(\z)
	&=\int_\D\left(\sum_{k=0}^\infty\widehat{f_r}(k)\z^k\right)
	\left(\sum_{j=0}^{\infty}\frac{(z\overline{\z})^j}{2{(\mu^n_\nu)_{2j+1}}}\right)\nu(\z)\,dA(\z)\\
  &=\sum_{k=0}^\infty\frac{\nu_{2k+1}}{(\mu^n_\nu)_{2k+1}}\widehat{f_r}(k)z^k
	=\frac{d^n}{dz^n}\left(z^n f_r(z)\right),\quad 0<r<1.
  \end{align*}
By letting $r\to1^-$, we obtain \eqref{id:generalizedreproducing}. The asymptotic equality \eqref{Ieq:newweight1} follows by the first identity in \eqref{jhgf}, and \eqref{Ieq:newweight} is an immediate consequence of \eqref{Ieq:newweight1} and \cite[Lemma~2.1]{PelSum14}.
\end{proof}

Lemma~\ref{nthreproducing} yields a novel representation of the small Hankel operator $h^\nu_f$ for every radial weight $\nu$. Namely, if the function $z\mapsto f(z)\|B_z^{\mu^n_\nu}\|_{A^1_{\mu^n_\nu}}$ belongs to $L^1_\nu$, then the reproducing formula for functions in $A^1_{\mu^n_\nu}$ together with Fubini's theorem and \eqref{id:generalizedreproducing} yields
	\begin{equation}\label{Id: newrepresentation}
  \begin{split}
  h_f^\nu(g)(z)
  &=\int_\D f(\z)\overline{\int_\D g(u)B_z^{\nu}(u)\overline{B_\z^{\mu^n_{\nu}}}(u)\mu^n_{\nu}(u)\,dA(u)}\nu(\z)\,dA(\z)\\
  &=\int_\D\overline{g(u)B_z^{\nu}(u)}\left(\int_\D f(\z)\overline{B^{\mu^n_{\nu}}_u(\z)}\nu(\z)\,dA(\z)\right)\mu^n_{\nu}(u)\,dA(u)\\
  &=\int_{\D}\overline{g(u)B_z^{\nu}(u)}\left(\frac{d^n}{du^n}\left(u^n f(u)\right)\right)\mu^n_{\nu}(u)\,dA(u),\quad g\in H^\infty.
  \end{split}
	\end{equation}

The representation \eqref{Id: newrepresentation} of the small Hankel operator $h^\nu_f$ involves the quantity $\frac{d^n}{d(\cdot)^n}\left((\cdot)^n f\right)$. In order to pass from this function to $f^{(n)}$ we will use the following two lemmas. The first one compares in a natural way the growth of the afore-mentioned functions in the weighted Bergman space $A^p_W$.

\begin{lemma}\label{Lemma:derivative}
Let $0<p<\infty$ and $n\in\N$, and let $W\not\equiv0$ be a radial weight. Then the function $z\mapsto F_{f,n}(z)=\frac{d^n}{dz^n}(z^nf(z))$ satisfies
	$$
	\|F_{f,n}\|_{A^p_W}\asymp\|f^{(n)}\|_{A^p_W}+\sum_{j=0}^{n-1}|f^{(j)}(0)|,\quad f\in\H(\D).
	$$
\end{lemma}

\begin{proof}
The non-tangential approach region with vertex at $\z\in\D\setminus\{0\}$ is defined by
    \begin{equation*}
    \Gamma(\z)=\left\{z\in \D:\,|\arg\zeta-\arg
    z|<\frac12\left(1-\left|\frac{z}{\zeta}\right|\right)\right\},
    \end{equation*}
and the non-tangential maximal function is $N(f)(\z)=\sup_{z\in\Gamma(\z)}|f(z)|$. It is known that each $0<p<\infty$ and each radial weight $\om$ satisfy
	\begin{equation}\label{eq:non-tangential norm}
    \|N(f)\|_{A^p_\om}\asymp\|f\|_{A^p_\om},\quad f\in\H(\D),
    \end{equation}
by \cite[Lemma~4.4]{PR}.

Let $n\in\N$ be fixed. By the Leibniz' rule
	\begin{equation}\label{Eq:Leibniz}
	\begin{split}
	\left|F_{f,n}(z)\right|
	=\left|\frac{d^n}{dz^n}(z^nf(z))\right|
	=\left|\sum_{j=0}^n\binom{n}{j}\left(\frac{d^{n-j}}{dz^{n-j}}z^n\right)f^{(j)}(z)\right|
	\lesssim\sum_{j=0}^n\left|f^{(j)}(z)\right|,\quad z\in\D.
	\end{split}
	\end{equation}
For each $g\in\H(\D)$ we have
	$$
	|g(z)|=\left|\int_0^zg'(\zeta)\,d\zeta+g(0)\right|\le\max_{0\le r\le|z|}\left|g'\left(r\frac{z}{|z|}\right)\right|+|g(0)|
	\le N(g')(z)+|g(0)|,\quad z\in\D\setminus\{0\},
	$$
and hence
	\begin{equation}\label{Eq:Leibnizx}
	\left|f^{(j)}(z)\right|\le N\left(f^{(n)}\right)(z)+\sum_{k=j}^{n-1}|f^{(k)}(0)|,\quad j=0,\ldots, n-1.
	\end{equation}
By combining \eqref{Eq:Leibniz} and \eqref{Eq:Leibnizx} we deduce
	$$
	\left|F_{f,n}(z)\right|\lesssim N\left(f^{(n)}\right)(z)+\sum_{j=0}^{n-1}|f^{(j)}(0)|,\quad z\in\D\setminus\{0\},
	$$
from which \eqref{eq:non-tangential norm} yields
	$$
	\|F_{f,n}\|_{A^p_W}\lesssim\|f^{(n)}\|_{A^p_W}+\sum_{j=0}^{n-1}|f^{(j)}(0)|,\quad f\in\H(\D).
	$$

For the converse inequality, we observe that straightforward calculations based on \eqref{Eq:Leibnizx} give
	\begin{equation*}
	\begin{split}
	|f^{(n)}(z)|
	\lesssim|z^nf^{(n)}(z)|
	\lesssim|F_{f,n}(z)|+\sum_{j=0}^{n-1}|z^jf^{(j)}(z)|
	\lesssim N\left(F_{f,n}\right)(z),\quad |z|\ge\frac12.
	\end{split}
	\end{equation*}
This estimate together with \eqref{eq:non-tangential norm} gives
	$$
	\|f^{(n)}\|_{A^p_W}
	\lesssim\left(\int_{\D\setminus D(0,\frac12)}|f^{(n)}(z)|^pW(z)\,dA(z)\right)^\frac1p
	\lesssim\|F_{f,n}\|_{A^p_W}.
	$$
Further, since $W$ is a radial weight, we have
	$$
	|f^{(j)}(0)|^p\lesssim\|f^{(j)}\|_{A^p_W}^p
	\lesssim\int_{\D\setminus D(0,\frac12)}|f^{(j)}(z)|^pW(z)\,dA(z)
	\lesssim\int_{\D\setminus D(0,\frac12)}|z^jf^{(j)}(z)|^pW(z)\,dA(z)
	$$
for all $f\in\H(\D)$. Therefore the argument used above guarantees
	$$
	\sum_{j=0}^{n-1}|f^{(j)}(0)|\lesssim\|F_{f,n}\|_{A^p_W}.
	$$
This completes the proof of the lemma.
\end{proof}

The next lemma compares the growth of the maximum muduli of the functions $F_{f,n}$ and $f^{(n)}$.

\begin{lemma}\label{Lemma:derivative2}
Let $n\in\N$, and let $W:\D\to[0,\infty)$ be a bounded function such that $W\not\equiv0$ and $W(z)=W(|z|)$ for all $z\in\D$. Then the function $z\mapsto F_{f,n}(z)=\frac{d^n}{dz^n}(z^nf(z))$ satisfies
	$$
	\sup_{z\in\D}\left(|F_{f,n}(z)|W(z)\right)\asymp\sup_{z\in\D}\left(|f^{(n)}(z)|W(z)\right)+\sum_{j=0}^{n-1}|f^{(j)}(0)|,\quad f\in\H(\D).
	$$
\end{lemma}

\begin{proof}
Let $n\in\N$ be fixed. For each $g\in\H(\D)$ we have
	\begin{equation}\label{lslsododo}
	M_\infty(r,g)
	=\max_{|z|=r}\left|\int_0^zg'(\zeta)\,d\zeta+g(0)\right|
	\le M_\infty(r,g')+|g(0)|,
	\end{equation}
and hence \eqref{Eq:Leibniz} yields
	\begin{equation*}
	\begin{split}
	M_\infty\left(r,F_{f,n}\right)
	\lesssim M_\infty\left(r,f^{(n)}\right)+\sum_{j=0}^{n-1}|f^{(j)}(0)|,\quad 0<r<1.
	\end{split}
	\end{equation*}
It follows that
	$$
	\sup_{z\in\D}\left(|F_{f,n}(z)|W(z)\right)
	\lesssim\sup_{z\in\D}\left(|f^{(n)}(z)|W(z)\right)+\sum_{j=0}^{n-1}|f^{(j)}(0)|,\quad f\in\H(\D).
	$$
	
For the converse inequality, we observe that straightforward calculations based on \eqref{lslsododo} give
	\begin{equation*}
	\begin{split}
	M_\infty(r,f^{(n)})
	&\lesssim M_\infty(r,(\cdot)^nf^{(n)})
	\lesssim M_\infty\left(r,F_{f,n}\right)+\sum_{j=0}^{n-1}M_\infty\left(r,(\cdot)^jf^{(j)}\right)
	\lesssim M_\infty\left(r,F_{f,n}\right)
	\end{split}
	\end{equation*}
for all $\frac12\le r<1$. This estimate yields
	$$
	\sup_{z\in\D}\left(|f^{(n)}(z)|W(z)\right)\lesssim\sup_{z\in\D}\left(|F_{f,n}(z)|W(z)\right),\quad f\in\H(\D).
	$$
Further, since
	$$
	|f^{(j)}(0)|\le M_\infty\left(r,f^{(j)}\right)
	\lesssim M_\infty\left(r,(\cdot)^jf^{(j)}\right),\quad \frac12\le r<1,\quad f\in\H(\D),
	$$
the argument used above guarantees
	$$
	\sum_{j=0}^{n-1}|f^{(j)}(0)|\lesssim\sup_{z\in\D}\left(|F_{f,n}(z)|W(z)\right),\quad f\in\H(\D).
	$$
This completes the proof of the lemma.
\end{proof}

For $x\in\mathbb{R}$ and a function $\om$, defined on $\D$, write $\om_{[x]}(z)=\om(z)(1-|z|)^x$ for all $z\in\D$. Observe that, for each $0<p<\infty$ and $\om\in\DDD$, we have
	\begin{equation}\label{littlewoodpaley}
	\begin{split}
	\int_\D|f(z)|^p\om(z)\,dA(z)
	&\asymp\int_\D|f'(z)|^p\om_{[p]}(z)\,dA(z)+|f(0)|^p,\quad f\in\H(\D),
	\end{split}
	\end{equation}
by \cite[Theorem~5]{PelRat2020}.

\begin{lemma}\label{lemma:bloch-measure}
Let $0<p<\infty$, $n\in\N$ and $\om\in\DDD$. Then the function $z\mapsto F_{f,n}(z)=\frac{d^n}{dz^n}(z^nf(z))$ satisfies
	\begin{equation*}
	\begin{split}
	\sup_{h\in\B\setminus\{0\}}\frac{\int_\D |h(z)|^p|F_{f,n}(z)|^p\om(z)\,dA(z)}{\|h\|^p_\B}
	\quad\asymp\sup_{h\in\B\setminus\{0\}}\frac{\int_\D |h(z)|^p|f^{(n)}(z)|^p\om(z)\,dA(z)}{\|h\|^p_\B}+\sum_{j=0}^{n-1}|f^{(j)}(0)|^p
	\end{split}
	\end{equation*}
for all $f\in\H(\D)$.
\end{lemma}

\begin{proof}
Assume first that
	$$
	C_f=\sup_{h\in\B\setminus\{0\}}\frac{\int_\D |h(z)|^p|f^{(n)}(z)|^p\om(z)\,dA(z)}{\|h\|^p_\B}<\infty.
	$$
The estimate \eqref{Eq:Leibniz} yields
	\begin{equation}\label{eqx}
	\int_\D|h(z)|^p|F_{f,n}(z)|^p\om(z)\,dA(z)
	\lesssim\sum_{j=0}^{n}\int_\D|h(z)|^p|f^{(j)}(z)|^p\om(z)\,dA(z),\quad f\in\H(\D).
	\end{equation}
Further, by \eqref{littlewoodpaley}, we deduce
	\begin{equation*}
	\begin{split}
	\int_{\D} |h(z)|^p|f^{(j)}(z)|^p\om(z)\,dA(z)
	&\lesssim\int_{\D}\left|h'(z)f^{(j)}(z)\right|^p\om_{[p]}(z)\,dA(z)\\
	&\quad+\int_{\D}\left|h(z)f^{(j+1)}(z)\right|^p\om_{[p]}(z)\,dA(z)+|h(0)|^p|f^{(j)}(0)|^p\\
	&\lesssim\|h\|_\B^p\int_{\D}\left|f^{(j)}(z)\right|^p\om(z)\,dA(z)\\
	&\quad+\int_{\D}\left|h(z)f^{(j+1)}(z)\right|^p{\om}(z)\,dA(z)+|h(0)|^p|f^{(j)}(0)|^p\\
	&\lesssim\|h\|_\B^p\int_\D|f^{(j+1)}(z)|^p{\om}(z)\,dA(z)+\|h\|_\B^p|f^{(j)}(0)|^p\\
	&\quad+\int_{\D}\left|h(z)\right|^p|f^{(j+1)}(z)|^p{\om}(z)\,dA(z)\\
	&\lesssim\|h\|_\B^p\int_\D|f^{(n)}(z)|^p {\om}(z)\,dA(z)+\|h\|_\B^p\sum_{m=j}^{n-1}|f^{(m)}(0)|^p\\
	&\quad+\int_{\D}|h(z)|^p\left|f^{(j+1)}(z)\right|^p{\om}(z)\,dA(z)\\
	&\lesssim\left(C_f+\sum_{m=j}^{n-1}|f^{(m)}(0)|^p\right)\|h\|^p_\B,\quad h\in\B,\quad j=0,\ldots,n-1.
	\end{split}
	\end{equation*}
This together with \eqref{eqx} yields
	$$
	\int_\D |h(z)|^p|F_{f,n}(z)|^p\om(z)\,dA(z)
	\lesssim \|h\|_\B^p\left(C_f+\sum_{j=0}^{n-1}|f^{(j)}(0)|^p\right),\quad h\in\B.
	$$

Conversely, assume that
	$$
	C_F=\sup_{h\in\B\setminus\{0\}}\frac{\int_\D |h(z)|^p|F_{f,n}(z)|^p\om(z)\,dA(z)}{\|h\|^p_\B}<\infty.
	$$
By \eqref{Eq:Leibniz}, we have
	\begin{equation}\label{eqxxx}
	\begin{split}
	\int_\D |h(z)|^p|f^{(n)}(z)|^p\om(z)\,dA(z)
	&\lesssim\int_\D|h(z)|^p|F_{f,n}(z)|^p\om(z)\,dA(z)\\
	&\quad+\sum_{j=0}^{n-1}\int_\D |h(z)|^p|f^{(j)}(z)|^p\om(z)\,dA(z)\\
	&\le C_F\|h\|^p_\B+\sum_{j=0}^{n-1}\int_\D |h(z)|^p|f^{(j)}(z)|^p\om(z)\,dA(z).
	\end{split}
	\end{equation}
To estimate the second term, we use \eqref{littlewoodpaley}, the growth of Bloch functions to get
	\begin{equation*}
	\begin{split}
	\int_{\D}|h(z)|^p|f^{(j)}(z)|^p\om(z)\,dA(z)
	&\lesssim\int_\D|h'(z)f^{(j)}(z)|^p\om_{[p]}(z)\,dA(z)\\
	&\quad+\int_\D|h(z)f^{(j+1)}(z)|^p\om_{[p]}(z)\,dA(z)+|h(0)f^{(j)}(0)|^p\\
	&\lesssim\|h\|_\B^p\int_\D|f^{(j)}(z)|^p\om(z)\,dA(z)+|h(0)f^{(j)}(0)|^p\\
	&\quad+\|h\|^p_\B\int_\D|f^{(j+1)}(z)|^p\left(\log\frac{e}{1-|z|}\right)^p\om_{[p]}(z)\,dA(z)\\
	&\lesssim\|h\|_\B^p\int_\D|f^{(j+1)}(z)|^p\om_{[p]}(z)\,dA(z)+\|h\|_\B^p|f^{(j)}(0)|^p\\
	&\quad+\|h\|_\B^p\int_\D|f^{(j+1)}(z)|^p\om(z)\,dA(z)+|h(0)f^{(j)}(0)|^p\\
	&\lesssim\|h\|_\B^p\int_\D|f^{(j+1)}(z)|^p\om(z)\,dA(z)+\|h\|_\B^p|f^{(j)}(0)|^p\\
	&\asymp\|h\|_\B^p\left(\int_\D|f^{(j+1)}(z)|^p{\om}(z)\,dA(z)+|f^{(j)}(0)|^p\right).
	\end{split}
	\end{equation*}
By combining this estimate with \eqref{eqxxx} and by applying Lemma \ref{Lemma:derivative} we finally obtain
	\begin{align*}
	\int_\D |h(z)|^p|f^{(n)}(z)|^p\om(z)\,dA(z)
	&\lesssim\|h\|_\B^p\left(C_F+\int_\D |f^{(n)}(z)|^p\om(z)\,dA(z)+\sum_{j=0}^{n-1}|f^{(j)}(0)|^p\right)\\
	&\lesssim C_F\|h\|^p_\B,
	\end{align*}
which is what we wished to prove.
\end{proof}

\section{Proof of Theorem~\ref{theorem}}

Denote $\widetilde{\om}(z)=\frac{\widehat{\om}(z)}{1-|z|}$ for all $z\in\D$, and observe that $\|f\|_{A^p_\om}\asymp\|f\|_{A^p_{\widetilde{\om}}}$ for all $f\in\H(\D)$ by \cite[Proposition~5]{PRS1}, provided $\om\in\DDD$. The weight $\widetilde\om$ appears repeatedly in our arguments. Apart from the just mentioned asymptotic norm equality, $\widetilde\om$ has the obvious and useful property that it never vanishes on $\D$ as long as $\widehat{\om}$ is strictly positive which is our case by our initial hypothesis.

The next lemma shows that we may multiply $\widetilde{\om}$ by the distance from the boundary to a sufficiently small negative power depending on $\om\in\DDD$, and the resulting function is also a radial weight in the class $\DDD$.

\begin{lemma}\label{lemma:weights-simple}
Let $\om\in\DDD$. Then there exists an $\e=\e(\om)>0$ such that
	\begin{equation}\label{hyhyhyh}
	\int_r^1\widetilde{\om}_{[-\e]}(t)\,dt\asymp\frac{\widehat{\om}(r)}{(1-r)^{\e}},\quad 0\le r<1,
	\end{equation}
and thus $\widetilde{\om}_{[-\e]}\in\RR$.
\end{lemma}

\begin{proof}
By \cite[(2.27)]{PelRat2020}, $\om\in\Dd$ if and only if there exist constants $C=C(\om)>0$ and $\a=\a(\om)>0$ such that
	\begin{equation}\label{Eq:Dd-characterization}
	\widehat{\om}(t)\le C\left(\frac{1-t}{1-r}\right)^\a\widehat{\om}(r),\quad 0\le r\le t<1.
	\end{equation}
By choosing $\e=\e(\om)\in(0,\alpha)$, this and the hypothesis $\om\in\DD$ yield
	\begin{equation*}
	\begin{split}
	\int_r^1\widetilde{\om}_{[-\e]}(t)\,dt
	&\lesssim\frac{\widehat{\om}(r)}{(1-r)^\alpha}\int_r^1\frac{dt}{(1-t)^{1+\e-\alpha}}
	\asymp\frac{\widehat{\om}(r)}{(1-r)^\e}
	\lesssim\int_r^{1-\frac{1-r}{K}}\widetilde{\om}_{[-\e]}(t)\,dt
	\le\int_r^1\widetilde{\om}_{[-\e]}(t)\,dt
	\end{split}
	\end{equation*}
for each fixed $K>1$. Therefore \eqref{hyhyhyh} is satisfied, and $\widetilde{\om}_{[-\e]}\in\RR$ for all $\e\in(0,\alpha)$.
\end{proof}

With these preparations we are ready to show that \eqref{Eq:theorem} is a sufficient condition for $h^\nu_f: A^p_{\om}\rightarrow A^q_{\nu}$ to be bounded. The argument we emply will be also used in the proof of Theorem~\ref{theorem1}.

\medskip

\noindent\emph{Proof of sufficiency}. Assume \eqref{Eq:theorem}.
By Lemma~\ref{lemma:weights-simple} we may choose $\varepsilon=\varepsilon(\nu,q)>0$ sufficiently small such that both $\widetilde{\nu}_{[-(\varepsilon\frac{q'}{q})]}$ and $\widetilde{\nu}_{[-(\varepsilon)]}$ belong to $\DDD$. By \cite[Proposition~5]{PRS1}, Lemma~\ref{nthreproducing}, H\"older's inequality, Fubini's theorem, \eqref{Id: newrepresentation} and two applications of \cite[Theorem 1]{PR2014}, we deduce
\begin{align*}
  \|h^\nu_f(g)\|^q_{A^q_\nu}
	&\asymp\|h^\nu_f (g)\|^q_{A^q_{\widetilde{\nu}}}
	\le\int_\D\left(\int_{\D}|{g(\z)|\left|\frac{d^n}{d\z^n}(\z^nf(\z))\right||B_z^{\nu}(\z)}|\mu^n_{\nu}(\z)\,dA(\z)\right)^q\widetilde{\nu}(z)\,dA(z)\\
  &\asymp\int_\D\left(\int_{\D}|g(\z)|\left|\frac{d^n}{d\z^n}(\z^nf(\z))\right|(1-|\z|)^{n+\frac{\varepsilon}{q}}|B_z^{\nu}(\z)|^{\frac1q}\frac{|B_z^{\nu}(\z)|^{\frac1{q'}}}{(1-|\z|)^{\frac{\varepsilon}{q}}}\widetilde{\nu}(\z)\,dA(\z)\right)^q\widetilde{\nu}(z)\,dA(z)\\
  &\leq\int_\D\left(\int_\D|g(\z)|^q\left|\frac{d^n}{d\z^n}(\z^nf(\z))\right|^q(1-|\z|)^{nq+\varepsilon}|B^\nu_z(\z)|\widetilde{\nu}(\z)\,dA(\z)\right)\\
  &\quad\cdot\left(\int_\D|B^\nu_z(\z)|\widetilde{\nu}_{[-(\varepsilon\frac{q'}{q})]}(\zeta)\,dA(\z)\right)^{\frac{q}{q'}}\widetilde{\nu}(z)\,dA(z)\\
  &\asymp\int_\D\left(\int_\D|g(\z)|^q\left|\frac{d^n}{d\z^n}(\z^nf(\z))\right|^q(1-|\z|)^{nq+\varepsilon}|B^\nu_z(\z)|\widetilde{\nu}(\z)\,dA(\z)\right)\\
  &\quad\cdot\left(\int_0^{|z|}\frac{dt}{(1-t)^{1+\frac{q'\varepsilon}{q}}}+1\right)^{\frac{q}{q'}}\widetilde{\nu}(z)\,dA(z)\\
  &\asymp\int_\D\left(\int_\D|g(\z)|^q\left|\frac{d^n}{d\z^n}(\z^nf(\z))\right|^q(1-|\z|)^{nq+\varepsilon}|B^\nu_z(\z)|\widetilde{\nu}(\z)\,dA(\z)\right)\widetilde{\nu}_{[-\e]}(z)\,dA(z)\\
  &=\int_\D |g(\z)|^q\left|\frac{d^n}{d\z^n}(\z^nf(\z))\right|^q(1-|\z|)^{nq+\varepsilon}\widetilde{\nu}(\z)\left(\int_{\D}|B^\nu_\z(z)|\widetilde{\nu}_{[-\e]}(z)\,dA(z)\right)\,dA(\z)\\
  &\asymp\int_\D |g(\z)|^q\left|\frac{d^n}{d\z^n}(\z^nf(\z))\right|^q(1-|\z|)^{nq}\widetilde{\nu}(\z)\,dA(\z)\\
  &\leq\left(\int_\D|g(\z)|^p\widetilde{\om}(\z)\,dA(\z)\right)^{\frac{q}{p}}\cdot\left(\int_\D\left|\frac{d^n}{d\z^n}(\z^nf(\z))\right|^{\frac{pq}{p-q}}(1-|\z|)^{\frac{npq}{p-q}}\left(\frac{\widetilde{\nu}(\z)}{\widetilde{\om}(\z)^{\frac qp}}\right)^{\frac{p}{p-q}}\,dA(\z)\right)^{\frac{p-q}{p}}\\
  &\asymp\|g\|^q_{A^p_\om}\left(\int_\D\left|\frac{d^n}{d\z^n}(\z^nf(\z))(1-|\z|)^n\right|^{\frac{pq}{p-q}}\left(\frac{\widehat{\nu}(\z)}{\whw(\z)}\right)^{\frac{p}{p-q}}\widetilde{\om}(\z)\,dA(\z)\right)^{\frac{p-q}{p}},\quad g\in H^\infty.
\end{align*}
Since $\om$ is radial, $H^\infty$ is dense in $A^p_\om$, and therefore a standard density argument via the BLT-theorem, see for example \cite[Theorem~I.7]{ReedSimon}, now yields
	$$
	\|h^\nu_f\|^{\frac{pq}{p-q}}_{A^p_\om\rightarrow A^{q}_{\nu}}
	\lesssim
	\int_\D\left|\frac{d^n}{dz^n}(z^nf(z))(1-|z|)^n\right|^{\frac{pq}{p-q}}\left(\frac{\widehat{\nu}(z)}{\whw(z)}\right)^{\frac{p}{p-q}}\widetilde{\om}(z)\,dA(z)
	$$
for each $n\in\N$. Write
	\begin{equation}\label{def:W}
	W(z)=W_{\om,\nu,p,q,n}(z)=(1-|z|)^{n\frac{pq}{p-q}}\left(\frac{\widehat{\nu}(z)}{\whw(z)}\right)^{\frac{p}{p-q}}\widetilde{\om}(z),\quad z\in\D,
	\end{equation}
for short. Since $\om,\nu\in\DDD$, \cite[Lemma~2.1]{PelSum14} and \eqref{Eq:Dd-characterization} show that there exists an $N_1=N_1(\om,\nu,p,q)\in\N$ such that $W$ is a regular weight for each $n\ge N_1$. In particular, $W$ is a weight and hence Lemma~\ref{Lemma:derivative} now yields
	$$
	\|h^\nu_f\|_{A^p_\om\rightarrow A^{q}_{\nu}}
	\lesssim\|f^{(n)}\|_{A^{\frac{pq}{p-q}}_W}+\sum_{j=0}^{n-1}|f^{(j)}(0)|.
	$$
Therefore $h^\nu_f:A^p_{\om}\rightarrow A^q_{\nu}$ is bounded and the operator norm obeys the claimed upper bound for each fixed $n\ge N_1$. \hfill$\Box$

\medskip

For the necessity we will need a lemma concerning weights in $\DDD$. This one shows that a certain weight, induced by $\om,\nu\in\DDD$, appearing in our argument is regular.

\begin{lemma}\label{le:regularity}
Let $0<p\le 1$ and $\om,\nu\in\DDD$. Then $\left(\frac{\widehat{\nu}}{\widehat{\om}}\right)^p\widetilde{\om}$ is regular.
\end{lemma}

\begin{proof}
Since $\om,\nu\in\DDD\subset\Dd$, it follows from \eqref{Eq:Dd-characterization} that there exist $\beta_1=\beta_1(\nu)>0$ and $\beta_2=\beta_2(\om)>0$ such that
 \begin{align*}
     \int_r^1\left(\frac{\widehat{\nu}(s)}{\widehat{\om}(s)}\right)^p\widetilde{\om}(s)\,ds&=\int_r^1\frac{\widehat{\nu}(s)^p\whw(s)^{1-p}}{1-s}\,ds\\
     &\lesssim\left(\frac{\widehat{\nu}(r)}{(1-r)^{\beta_1}}\right)^p\left(\frac{\whw(r)}{(1-r)^{\beta_2}}\right)^{1-p}\int_r^1(1-s)^{\beta_1p+\beta_2(1-p)-1}\,ds\\
     &\asymp \widehat{\nu}(r)^p\whw(r)^{1-p},\quad 0\leq r<1.
 \end{align*}
Conversely, since $\om,\nu\in\DDD\subset\DD$, a reasoning similar to that above, but using \cite[Lemma~2.1]{PelSum14} instead of \eqref{Eq:Dd-characterization}, shows that there exist $\alpha_1=\alpha_1(\nu)>0$ and $\alpha_2=\alpha_2(\om)>0$ such that
  \begin{align*}
     \int_r^1\left(\frac{\widehat{\nu}(s)}{\widehat{\om}(s)}\right)^p\widetilde{\om}(s)\,ds
     &\gtrsim\left(\frac{\widehat{\nu}(r)}{(1-r)^{\alpha_1}}\right)^p\left(\frac{\whw(r)}{(1-r)^{\alpha_2}}\right)^{1-p}\int_r^1(1-s)^{\alpha_1p+\alpha_2(1-p)-1}\,ds\\
     &\asymp \widehat{\nu}(r)^p\whw(r)^{1-p},\quad 0\leq r<1.
 \end{align*}
From these estimates it follows that $\left(\frac{\widehat{\nu}}{\widehat{\om}}\right)^p\widetilde{\om}$ is a regular weight.
\end{proof}

The proof of the necessity relies also heavily on a so-called atomic decomposition of functions in $A^p_\om$ with $\om\in\DDD$. To this end, some notation used in \cite{PRS-2021} are needed. For each $K\in\mathbb{N}\backslash\{1\}$, $j\in\mathbb{N}\cup\{0\}$ and $l=0,1,\ldots,K^{j+3}-1$, the dyadic polar rectangle is defined as
	$$
	Q_{j,l}=\left\{z\in\mathbb{D}:~\rho_{j}\leq|z|<\rho_{j+1},\arg z\in\left[\frac{2\pi l}{K^{j+3}},\frac{2\pi (l+1)}{K^{j+3}}\right)\right\},\quad \rho_j=\rho_j(K)=1-K^{-j},
	$$
and its center is denoted by $\zeta_{j,l}$. For each $M \in \mathbb{N}$ and $k=1, \ldots, M^{2}$, the rectangle $Q_{j, l}^{k}$ is defined as the result of dividing $Q_{j, l}$ into $M^{2}$ pairwise disjoint rectangles of equal Euclidean area, and the centers of these squares are denoted by $\zeta_{j, l}^{k}$, respectively. Write $\lambda=\left\{\lambda_{j,l}^k\right\} \in \ell^{p}$ if
$$
\|\lambda\|_{\ell^p}=\left(\sum_{j=0}^{\infty}\sum_{l=0}^{K^{j+3}-1} \sum_{k=1}^{M^{2}}\left|\lambda_{j, l}^k\right|^{p}\right)^{\frac{1}{p}}<\infty .
$$

\medskip

\noindent\emph{Proof of necessity}. Assume that $h^\nu_f:A^p_{\om}\rightarrow A^q_{\nu}$ is bounded. The first step in the proof consists of showing that this implies the boundedness of another operator which is more convenient for testing. To find out what this operator is, let $B^\alpha_z$ denote the reproducing kernel of $A^2_\alpha$ with $-1<\alpha<\infty$ at $z$. Further, write $dA_\alpha(z)=(1-|z|^2)^\alpha dA(z)$ for short. Since $f\in A^1_{\nu_{\log}}$ by the hypothesis, two applications of Fubini's theorem and the reproducing formula for functions in $A^1_\alpha$ yield
	\begin{equation}\label{jhgjhgjhg}
	\begin{split}
	\langle h_{f}^{\nu}(g),h\rangle_{A_{\nu}^{2}}	
	&=\int_{\mathbb{D}}f(\zeta)\overline{g(\zeta)}\overline{h(\zeta)}\nu(\zeta)dA(\zeta)\\
	&=\int_{\mathbb{D}}f(\zeta)\overline{g(\zeta)}\overline{\int_{\mathbb{D}}h(z)\overline{B_{\zeta}^{\alpha}(z)}\,dA_{\alpha}(z)}\nu(\zeta)\,dA(\zeta)\\
	&=\int_{\mathbb{D}}\left(\int_{\mathbb{D}}f(\zeta)\overline{g(\zeta)}\overline{B_{z}^{\alpha}(\zeta)}\nu(\zeta)\,dA(\zeta)\right)\overline{h(z)}\,dA_{\alpha}(z)\\
	&=\langle S_{f}^{\nu,\alpha}(g),h\rangle_{A_{\alpha}^{2}}, \quad g,h\in H^\infty,
	\end{split}
	\end{equation}
where
	\begin{equation}\label{new-operator}
	S_{f}^{\nu,\alpha}(g)(z)=\int_{\mathbb{D}}f(\zeta)\overline{g(\zeta)}\overline{B_{z}^{\alpha}(\zeta)}\nu(\zeta)dA(\zeta),
	\quad z\in\mathbb{D},\quad g\in H^{\infty}.
	\end{equation}
By \eqref{jhgjhgjhg}, H\"{o}lder's inequality, the boundedness of $h^\nu_f$ and \cite[Proposition~5]{PRS1}, we deduce
	\begin{equation}\label{eqq-1}
	\begin{split}
	|\langle S_{f}^{\nu,\alpha}(g),h\rangle_{A_{\alpha}^{2}}|&=|\langle h_{f}^{\nu}(g),h\rangle_{A_{\nu}^{2}}|
	\le\| h_{f}^{\nu}(g)\|_{A_{\nu}^{q}}\|h\|_{A_{\nu}^{q'}}\\
	&\lesssim \|h_{f}^{\nu}\|_{A_{\omega}^{p}\to A_{\nu}^{q}}\|g\|_{A_{\omega}^{p}}\|h\|_{A_{\widetilde{\nu}}^{q'}},\quad g,h\in H^\infty.
	\end{split}
	\end{equation}
Since $\nu\in\DDD$ by the hypothesis, a standard argument based on \cite[Lemma~2.1]{PelSum14} reveals that there exists an $\alpha_0=\alpha_0(\nu,q)>0$ such that the function $V$, defined by
	$$
	V(z)
	=V_{\nu,q,\alpha}(z)
	=\left(\frac{(1-|z|^{2})^{\alpha}}{\widetilde{\nu}(z)}\right)^{q}\widetilde{\nu}(z),\quad z\in\mathbb{D},
	$$
is a regular weight for each $\alpha\ge\alpha_0$. Fix such an $\alpha$. Then $P_{\alpha}: L^{q'}_{\widetilde{\nu}}\to A^{q'}_{\widetilde{\nu}}$ is bounded by \cite[Theorem~3]{PR2014}, and further \cite[Theorem 6]{PR2014} yields $(A_{V}^{q})^{\star}\simeq A_{\widetilde{\nu}}^{q'}$ via the $A_{\alpha}^{2}$-pairing. It follows from this duality relation and \eqref{eqq-1} that
	\begin{equation}\label{boundednessofS}
	\|S_{f}^{\nu,\alpha}(g)\|_{A_{V}^{q}}
	\lesssim\|h_{f}^{\nu}\|_{A_{\omega}^{p}\to A_{\nu}^{q}}\|g\|_{A_{\omega}^{p}}, \quad g\in H^\infty,
	\end{equation}
and thus $S_{f}^{\nu,\alpha}:A_{\omega}^{p}\to A_{V}^{q}$ is bounded by the density argument.

The next step is to find suitable test functions. To do this, define
	\begin{equation}\label{def:U}
	U(z)
	=U_{\om,\nu,p,q}(z)
	=\left(\frac{\widehat{\nu}(z)}{\widehat{\omega}(z)}\right)^{\frac{pq-p}{pq-p+q}}\widetilde{\omega}(z),\quad z\in\mathbb{D}.
	\end{equation}
Since $1<q<p<\infty$ by the hypothesis, $\frac{pq-p}{pq-p+q}<1$, and therefore $U$ is a regular weight by Lemma~\ref{le:regularity}. Further, since $\om\in\DDD$ by the hypothesis, there exist $K_1=K_1(\om)\in\N$ and $K_2=K_2(U)\in\N$ such that \eqref{Dcheck} for $\om$ and $U$, respectively, is satisfied. Let $K=\max\{K_1,K_2\}$. For each $j\in\N\cup\{0\}$ and $l=0,1,\ldots,K^{j+3}-1$, let $Q_{j,l}$ be the corresponding dyadic polar rectangle. It follows from \cite[Theorem~1]{PRS-2021} that there exists an $M_{0}=M_0(\om,p)>0$ such that for each $M\ge M_0$ and each sequence $\{\lambda_{j,l}^{k}\}\in\ell^{p}$, the function $F$, defined by
	$$
	F(z)
	=F_{\om,p}(z)
	=\sum_{j,l,k}\lambda_{j,l}^{k}\frac{(1-|\z_{j,l}^{k}|)^{M-\frac{1}{p}}
	\widehat{\omega}(\z_{j,l}^{k})^{-\frac{1}{p}}}{(1-\overline{\z_{j,l}^{k}}z)^{M}}
	=\sum_{j,l,k}\lambda_{j,l}^{k}F_{j,l,k}(z),\quad z\in\D,
	$$
belongs to $A^p_\om$ and satisfies
	\begin{equation}\label{ie:atomic}
	\|F\|_{A^p_\om}\lesssim\|\{\lambda_{j,l}^{k}\}\|_{\ell^{p}}.
	\end{equation}
Here
	\begin{equation}\label{def-F}
	F_{j,l,k}(z)
	=\frac{(1-|\z_{j,l}^{k}|)^{M-\frac{1}{p}}
	\widehat{\omega}(\z_{j,l}^{k})^{-\frac{1}{p}}}{(1-\overline{\z_{j,l}^{k}}z)^{M}},\quad z\in\D,
	\end{equation}
and $\z_{j,l}^{k}$ is the center of the rectangle $Q^k_{j,l}$ for $k=1,2,\ldots,M^2$. The function $F$ is the starting point for the testing.

Recall that $S_{f}^{\nu,\alpha}:A_{\omega}^{p}\to A_{V}^{q}$ is bounded. The norm inequality \eqref{boundednessofS} together with \eqref{ie:atomic} yields
	\begin{equation}\label{hgfdhgfdhgfd}
	\begin{split}
  \|S_{f}^{\nu,\alpha}(F)\|^q_{A^q_V}
	&=\int_{\mathbb{D}}\bigg|\sum_{j,l,k}\lambda_{j,l}^{k}S_{f}^{\nu,\alpha}(F_{j,l,k})(z)\bigg|^{q}V(z)\,dA(z)\\
  &\lesssim\|h_{f}^{\nu}\|_{A_{\omega}^{p}\to A_{\nu}^{q}}^{q}\|F\|_{A_{\omega}^{p}}^{q}
	\lesssim\|h_{f}^{\nu}\|_{A_{\omega}^{p}\to A_{\nu}^{q}}^{q}\|\{\lambda_{j,l}^{k}\}\|_{\ell^{p}}^{q}.
	\end{split}
	\end{equation}
Let $r_{j,l}^{k}$ denote the Rademacher functions defined on $(0,1)$. We next replace $\lambda_{j,l}^{k}$ by $\lambda_{j,l}^{k}r_{j,l}^{k}(t)$ in \eqref{hgfdhgfdhgfd}, then integrate the obtained inequality with respect to $t$ from 0 to 1, and finally apply Fubini's theorem and Khinchine's inequality \cite[Appendix~A]{D} to get
	\begin{equation}\label{eqq-9}
	\int_{\mathbb{D}}
	\left(\sum_{j,l,k}|\lambda_{j,l}^{k}|^{2}|S_{f}^{\nu,\alpha}(F_{j,l,k})(z)|^{2}\right)^{\frac{q}{2}}V(z)\,dA(z)
	\lesssim\|h_{f}^{\nu}\|_{A_{\omega}^{p}\to A_{\nu}^{q}}^{q}\|\{\lambda_{j,l}^{k}\}\|_{\ell^{p}}^{q}.
	\end{equation}
To proceed further, some more notation is needed. The pseudohyperbolic disc of center $z\in\D$ and of radius $r\in(0,1)$ is denoted by $\Delta(z,r)$, and it is the set $\left\{\zeta\in\mathbb{D}:\left|\frac{z-\zeta}{1-\overline{z}\zeta}\right|<r\right\}$. It is well known that $\Delta(z,r)$ is the Euclidean disc centered at $\left(1-r^{2}\right) z /\left(1-r^{2}|z|^{2}\right)$ and of radius $\left(1-|z|^{2}\right)r/\left(1-r^{2}|z|^{2}\right)$. We claim that for each $1<q<\infty$,
	\begin{equation}\label{eqq-10}
	\begin{split}
	&\quad\sum_{j,l,k}|\lambda_{j,l}^{k}|^{q}\int_{\Delta(\z_{j,l}^{k},\delta)}\big|S_{f}^{\nu,\alpha}(F_{j,l,k})(z)\big|^{q}V(z)\,dA(z)\\
	&\lesssim
	\int_{\mathbb{D}}\bigg(\sum_{j,l,k}|\lambda_{j,l}^{k}|^{2}|S_{f}^{\nu,\alpha}(F_{j,l,k})(z)|^{2}\bigg)^{\frac{q}{2}}V(z)\,dA(z).
	\end{split}
	\end{equation}
To see this, first observe that $\{\z_{j,l}^{k}\}$ is a $\delta$-lattice for some $\delta=\delta(\om,\nu)\in (0,1)$ by the definition. Hence, if $q\geq2$, then the inequality $\sum_{j} c_{j}^{x} \leq\left(\sum_{j} c_{j}\right)^{x}$, valid for all $c_{j} \geq 0$ and $x \geq 1$, and Fubini's theorem imply
	\begin{align*}
   &\quad\sum_{j,l,k}|\lambda_{j,l}^{k}|^{q}\int_{\Delta(\z_{j,l}^{k},\delta)}\big|S_{f}^{\nu,\alpha}(F_{j,l,k})(z)\big|^{q}V(z)\,dA(z)\\
   &=\int_\D\bigg(\sum_{j,l,k}|\lambda_{j,l}^{k}|^{q}|S_{f}^{\nu,\alpha}(F_{j,l,k})(z)|^{q}\chi_{\Delta(\z_{j,l}^{k},\delta)}(z)\bigg)^{\frac2q\cdot\frac{q}{2}}V(z)\,dA(z)\\
   &\leq\int_\D\bigg(\sum_{j,l,k}|\lambda_{j,l}^{k}|^{2}|S_{f}^{\nu,\alpha}(F_{j,l,k})(z)|^{2}\bigg)^{\frac{q}{2}}V(z)\,dA(z).
\end{align*}
To get the same estimate for $1<q<2$ we apply H\"older's inequality. It together with the fact that the number of discs $\Delta(\z_{j,l}^{k},\delta)$ to which each $z$ may belong to is uniformly bounded yields
	\begin{align*}
  &\quad\sum_{j,l,k}|\lambda_{j,l}^{k}|^{q}\int_{\Delta(\z_{j,l}^{k},\delta)}\big|S_{f}^{\nu,\alpha}(F_{j,l,k})(z)\big|^{q}V(z)\,dA(z)\\
  &\le\int_{\mathbb{D}}\left(\sum_{j,l,k}\big|\lambda_{j,l}^{k}S_{f}^{\nu,\alpha}(F_{j,l,k})(z)\big|^{2}\chi_{\Delta(\z_{j,l}^{k},\delta)}(z)\right)^{\frac{q}{2}}\cdot\left(\sum_{j,l,k}\chi_{\Delta(\z_{j,l}^{k},\delta)}(z)\right)
	^{1-\frac{q}{2}}V(z)\,dA(z)\\
	&\lesssim\int_\D\left(\sum_{j,l,k}|\lambda_{j,l}^{k}|^{2}|S_{f}^{\nu,\alpha}(F_{j,l,k})(z)|^{2}\right)
	^{\frac{q}{2}}V(z)\,dA(z).
	\end{align*}
Consequently, by combining \eqref{eqq-9}, \eqref{eqq-10}, the subharmonicity of $|S_{f}^{\nu,\alpha}(F_{j,l,k})|^{q}$ and the fact that $V$ is regular, we deduce
	\begin{equation}\label{Ie:sub-atomic}
  \begin{split}
  \|h_{f}^{\nu}\|_{A_{\omega}^{p}\to A_{\nu}^{q}}^{q}\|\{\lambda_{j,l}^{k}\}\|_{\ell^{p}}^{q}
	&\gtrsim \int_{\mathbb{D}}\bigg(\sum_{j,l,k}|\lambda_{j,l}^{k}|^{2}|S_{f}^{\nu,\alpha}(F_{j,l,k})(z)|^{2}\bigg)^{\frac{q}{2}}V(z)\,dA(z)\\
  &\gtrsim\sum_{j,l,k}|\lambda_{j,l}^{k}|^{q}\int_{\Delta(\z_{j,l}^{k},\delta)}\big|S_{f}^{\nu,\alpha}(F_{j,l,k})(z)\big|^{q}V(z)\,dA(z)\\
  &\gtrsim\sum_{j,l,k}|\lambda_{j,l}^{k}|^{q}V(\z_{j,l}^{k})(1-|\z_{j,l}^{k}|)^{2}|S_{f}^{\nu,\alpha}(F_{j,l,k})(\z_{j,l}^{k})|^{q}.
  \end{split}
	\end{equation}
Further, the definitions \eqref{new-operator} and \eqref{def-F}, together with the reproducing formula for functions in $A^1_{\mu_{\nu}^{n}}$, \eqref{id:generalizedreproducing} and Fubini's theorem yield
	\begin{align*}
  S_{f}^{\nu,\alpha}(F_{j,l,k})(\z_{j,l}^{k})
	=&\int_{\mathbb{D}}f(z)\overline{F_{j,l,k}(z) B^{\alpha}_{\z_{j,l}^{k}}(z)}\nu(z)\,dA(z)\\
	=&\frac{(1-|\z_{j,l}^{k}|)^{M-\frac{1}{p}}}{\widehat{\omega}(\z_{j,l}^{k})^{\frac{1}{p}}}
\int_{\mathbb{D}}f(z)\frac{1}{(1-\z_{j,l}^{k}\bar{z})^{M+\alpha+2}}\nu(z)\,dA(z)\\
=&\frac{(1-|\z_{j,l}^{k}|)^{M-\frac{1}{p}}}{\widehat{\omega}(\z_{j,l}^{k})^{\frac{1}{p}}}
\int_{\mathbb{D}}f(z)\overline{\int_{\mathbb{D}}\frac{1}{(1-\overline{\z_{j,l}^{k}}\zeta)^{M+\alpha+2}}
\overline{B_{z}^{\mu_{\nu}^{n}}(\zeta)}\mu_{\nu}^{n}(\zeta)\,dA(\zeta)}\nu(z)\,dA(z)\\
=&\frac{(1-|\z_{j,l}^{k}|)^{M-\frac{1}{p}}}{\widehat{\omega}(\z_{j,l}^{k})^{\frac{1}{p}}}
\int_{\mathbb{D}}\left(\int_{\mathbb{D}}f(z)\overline{B_{\zeta}^{\mu_{\nu}^{n}}(z)}\nu(z)\,dA(z)\right)\frac{\mu_{\nu}^{n}(\zeta)}{(1-\z_{j,l}^{k}\bar{\zeta})^{M+\alpha+2}}
\,dA(\zeta)\\
=&\frac{(1-|\z_{j,l}^{k}|)^{M-\frac{1}{p}}}{\widehat{\omega}(\z_{j,l}^{k})^{\frac{1}{p}}}
\int_{\mathbb{D}}F_{f,n}(\z)\frac{\mu_{\nu}^{n}(\zeta)}{(1-\z_{j,l}^{k}\bar{\zeta})^{M+\alpha+2}}\,dA(\zeta),
	\end{align*}
where $F_{f,n}(\z)=\frac{d^n}{d\z^n}(\z^nf(\z))$.
This together with \eqref{Ie:sub-atomic} gives
	\begin{equation*}
	\begin{split}
	&\sum_{j,l,k}|\lambda_{j,l}^{k}|^{q}V(\z_{j,l}^{k})(1-|\z_{j,l}^{k}|^{2})^{2+qM-\frac{q}{p}}\widehat{\omega}(\z_{j,l}^{k})^{-\frac{q}{p}}
\bigg|\int_{\mathbb{D}}\frac{F_{f,n}(\z)}{(1-\z_{j,l}^{k}\overline{\zeta})^{M+\alpha+2}}\mu_{\nu}^{n}(\zeta)\,dA(\zeta)\bigg|^{q}\\
	&\quad\lesssim\|h_{f}^{\nu}\|_{A_{\omega}^{p}\to A_{\nu}^{q}}^{q}\|\{\lambda_{j,l}^{k}\}\|_{\ell^{p}}^{q}.
	\end{split}
	\end{equation*}
Write now $s=\frac{pq}{p-q}$ for short, and recall that $1<q<p<\infty$ by the hypothesis. The well-known duality relation $(\ell^{\frac{p}{q}})^{\star}\simeq \ell^{\frac{p}{p-q}}$ now implies
	\begin{equation*}
	\sum_{j,l,k}V(\z_{j,l}^{k})^{\frac{s}{q}}\frac{(1-|\z_{j,l}^{k}|^{2})^{\frac{2s}{q}+sM-\frac{s}{p}}}
	{\widehat{\omega}(\z_{j,l}^{k})^{\frac{s}{p}}}
	\bigg|\int_{\mathbb{D}}\frac{F_{f,n}(\z)}{(1-\z_{j,l}^{k}\overline{\zeta})^{M+\alpha+2}}\mu_{\nu}^{n}(\zeta)\,dA(\zeta)\bigg|^{s}
	\lesssim\|h_{f}^{\nu}\|_{A_{\omega}^{p}\to A_{\nu}^{q}}^{s},
	\end{equation*}
which in turn gives
	\begin{equation}\label{Ie:dualrelation1}
	\limsup_{r\to1^-}\sum_{j,l,k}V(\z_{j,l}^{k})^{\frac{s}{q}}\frac{(1-|\z_{j,l}^{k}|^{2})^{\frac{2s}{q}+sM-\frac{s}{p}}}{\widehat{\omega}(\z_{j,l}^{k})^{\frac{s}{p}}}
\left|\int_{\mathbb{D}}\frac{\left(F_{f,n}\right)_r(\zeta)}{(1-\z_{j,l}^{k}\overline{\zeta})^{M+\alpha+2}}\mu_{\nu}^{n}(\zeta)\,dA(\zeta)\right|^{s}
	\lesssim\|h_{f}^{\nu}\|_{A_{\omega}^{p}\to A_{\nu}^{q}}^{s}.
	\end{equation}
To be precise, to get \eqref{Ie:dualrelation1} by means of Lebesgue's dominated convergence theorem, we must show that
\begin{align}\label{xxxxxx}
  \int_{\D}\left|F_{f,n}(\zeta)\right| \mu_{\nu}^{n}(\zeta)\,dA(\zeta)<\infty.
\end{align}
But $n$ applications of the inequality $M_1(r,f')\lesssim M_1(\frac{1+r}{2},f)/(1-r)$, which is a consequence of the Cauchy integral formula, together with the hypothesis $\nu\in\DDD$ easily shows that
		$$
		\int_\D \left|F_{f,n}(\zeta)\right|
	\log\frac{e}{1-|\zeta|}\mu^n_\nu(\z)\,dA(\z)
	\lesssim\int_\D\left|f(\z)\right|
	\log\frac{e}{1-|\zeta|}\widetilde{\nu}(\z)\,dA(\z),\quad f\in\H(\D).
		$$
Therefore, to deduce \eqref{xxxxxx}, it suffices to bound the last quantity by the norm of $f$ in $A^1_{\nu_{\log}}$. But this follows by an integration by parts similar to that in \cite[Proof of Theorem~6]{PelRat2020} once we show that
		$$
		\int_r^1\log\frac{e}{1-t}\widetilde{\nu}(t)\,dt\lesssim\int_r^1\log\frac{e}{1-t}\nu(t)\,dt,\quad 0\le r <1.
		$$
Since $\nu\in\DDD\subset\Dd$ by the hypothesis, there exists an $\alpha=\alpha(\nu)>0$ such that
	\begin{equation*}
	\begin{split}
  \int_r^1\log\frac{e}{1-t}\widetilde{\nu}(t)\,dt
	&\lesssim\frac{\widehat{\nu}(r)}{(1-r)^\alpha}\int_r^1\frac{\log\frac{e}{1-t}}{(1-t)^{1-\a}}\,dt\lesssim\widehat{\nu}(r)\log\frac{e}{1-r},\quad 0\le r<1.
	\end{split}
	\end{equation*}
Further, an integration by parts gives
	\begin{equation*}
	\begin{split}
	\int_r^1\log\frac{e}{1-t}\nu(t)\,dt\ge\widehat{\nu}(r)\log\frac{e}{1-r},\quad 0\le r <1,
	\end{split}
	\end{equation*}
and hence \eqref{xxxxxx} holds.

Recall that $W$ and $U$ are defined by \eqref{def:W} and \eqref{def:U}, respectively, $U$ is regular, and so is $W$ if $n\ge N_1$. Since also $\mu^n_\nu$ is regular, it follows from \cite[Theorem 3]{PR2014} that the Bergman projection $P_{\mu_{\nu}^{n}}: L^{s}_{W}\to A^{s}_{W}$ is bounded, and hence $(A_{U}^{s'})^{\star}\simeq A_{W}^{s}$ via the $A^2_{\mu_{\nu}^{n}}$-pairing by \cite[Theorem 6]{PR2014}. Moreover, by \cite[Theorem~2]{PRS-2021}, there exists an $M_1=M_1(\om,\nu,p,q)>M_0$ such that for each fixed $M>M_1$, each $h\in A^{s'}_U$ can be represented as
	\begin{equation}\label{rep:h}
	h(z)=\sum_{j,l,k}\gamma(h)_{j,l}^{k}\frac{(1-|\z_{j,l}^{k}|^{2})^{M+\alpha+2-\frac{1}{s'}}\widehat{U}(\rho_{j})^{-\frac{1}{s'}}}
	{(1-\overline{\z_{j,l}^{k}}z)^{M+\alpha+2}}
	\end{equation}
with
	\begin{equation}\label{lkjhgfhgfd}
	\|h\|_{A^{s'}_{U}}\asymp\|\{\gamma(h)_{j,l}^{k}\}\|_{\ell^{s'}},
	\end{equation}
where $\rho_j=1-K^{-j}$ for $K\in\N$ that has been fixed already. Therefore, the duality $(A_{U}^{s'})^{\star}\simeq A_{W}^{s}$, via the $A^2_{\mu_{\nu}^{n}}$-pairing, together with the representation \eqref{rep:h}, Fubini's theorem, H\"older's inequality and \eqref{lkjhgfhgfd} yields
	\begin{equation*}
  \begin{split}
	&\left\|\left(F_{f,n}\right)_r\right\|_{A_{W}^{s}}
	=\sup_{\{h:\|h\|_{A^{s'}_{U}}=1\}}\left|\int_{\mathbb{D}}h(z)\overline{\left(F_{f,n}\right)_r(z)}\mu_{\nu}^{n}(z)\,dA(z)\right|\\
	&=\sup_{\{h:\|h\|_{A^{s'}_{U}}=1\}}\left|\int_{\mathbb{D}}\sum_{j,l,k}\gamma(h)_{j,l}^{k}\frac{(1-|\z_{j,l}^{k}|^{2})^{M+\alpha+2-\frac{1}{s'}}
	\widehat{U}(\rho_{j})^{-\frac{1}{s'}}}{(1-\overline{\z_{j,l}^{k}}z)^{M+\alpha+2}}\overline{\left(F_{f,n}\right)_r(z)}\mu_{\nu}^{n}(z)\,dA(z)\right|\\
	&=\sup_{\{h:\|h\|_{A^{s'}_{U}}=1\}}\left|\sum_{j,l,k}\gamma(h)_{j,l}^{k}\frac{(1-|\z_{j,l}^{k}|^{2})^{M+\alpha+2-\frac{1}{s'}}}{\widehat{U}(\rho_{j})^{\frac{1}{s'}}}
	\int_{\mathbb{D}}\frac{\overline{\left(F_{f,n}\right)_r(z)}}{(1-\overline{\z_{j,l}^{k}}z)^{M+\alpha+2}}\mu_{\nu}^{n}(z)\,dA(z)\right|\\
	&\le\sup_{\{h:\|h\|_{A^{s'}_{U}}=1\}}\|\{\gamma(h)_{j,l}^{k}\}\|_{\ell^{s'}}\\
	&\quad\cdot\left(\sum_{j,l,k}\frac{(1-|\z_{j,l}^{k}|^{2})^{sM+\alpha s+2s-\frac{s}{s'}}}{\widehat{U}(\rho_{j})^{\frac{s}{s'}}}
\left|\int_{\mathbb{D}}\left(F_{f,n}\right)_r(z)\frac{\mu_{\nu}^{n}(z)}{(1-\z_{j,l}^{k}\bar{z})^{M+\alpha+2}}\,dA(z)\right|^{s}\right)^{\frac{1}{s}}\\
&\asymp\left(\sum_{j,l,k}\frac{(1-|\z_{j,l}^{k}|^{2})^{sM+\alpha s+2s-\frac{s}{s'}}}{\widehat{U}(\rho_{j})^{\frac{s}{s'}}}
\left|\int_{\mathbb{D}}\left(F_{f,n}\right)_r(z)\frac{\mu_{\nu}^{n}(z)}{(1-\z_{j,l}^{k}\bar{z})^{M+\alpha+2}}\,dA(z)\right|^{s}\right)^{\frac{1}{s}},\quad 0<r<1,
\end{split}
\end{equation*}
provided $n$ is sufficiently large, say $n\ge N_2$, where $N_2=N_2(\om,\nu,p,q)\in\N$. Observe that the use of the dilation in $F_{f,n}$ in the formula above guarantees that the order of the summation and the integration can be interchanged on the third line. This is attributed to
	\begin{equation}\label{fubinixxx}
    \sum_{j,l,k}\frac{(1-|\z_{j,l}^{k}|^{2})^{sM+\alpha s+2s-\frac{s}{s'}}}{\widehat{U}(\rho_{j})^{\frac{s}{s'}}}
	\left(\int_{\mathbb{D}}\frac{\mu_{\nu}^{n}(z)}{|1-\z_{j,l}^{k}\bar{z}|^{M+\alpha+2}}\,dA(z)\right)^{s}<\infty,
	\end{equation}
provided $n$ is sufficiently large, depending on $\om,\nu,p$ and $q$.

Indeed, on one hand, for a fixed $c>1$, since $U$ is regular, \cite[Lemma~2.1]{PelSum14} allows us to find an $M_2=M_2(\om,\nu,p,q,c)>M_1$ such that for each fixed $M\ge M_2$ we have
		\begin{equation}\label{fubinixxx1}
    \frac{(1-|z|^{2})^{sM+\alpha s+2s-\frac{s}{s'}}}{\widehat{U}(z)^{\frac{s}{s'}}}\lesssim(1-|z|^{2})^c,\quad z\in\mathbb{D}.
		\end{equation}
On the other hand, it follows from \eqref{Ieq:newweight} and \cite[Theorem 1]{PR2014} that for each $n>M+\alpha+1$ we have
	\begin{equation}\label{fubinixxx2}
	\int_{\mathbb{D}}\frac{\mu_{\nu}^{n}(z)}{|1-\z_{j,l}^{k}\overline{z}|^{M+\alpha+2}}\,dA(z)
	\lesssim1,\quad j\in\mathbb{N}\cup\{0\},\quad l=0,1,\ldots,K^{j+3}-1,\quad k=1,2,\ldots,M^{2}.
	\end{equation}
Now, since $\{\z^k_{j,l}\}$ is $\delta$-lattice, by combining \eqref{fubinixxx1} and \eqref{fubinixxx2} we deduce that for each fixed $M\ge M_2$ and any $n>N_2=[M+\alpha+1]+1$,
	\begin{equation*}
	\begin{split}
	&\sum_{j,l,k}\frac{(1-|\z_{j,l}^{k}|^{2})^{sM+\alpha s+2s-\frac{s}{s'}}}{\widehat{U}(\rho_{j})^{\frac{s}{s'}}}
	\left(\int_{\mathbb{D}}\frac{\mu_{\nu}^{n}(z)}{|1-\z_{j,l}^{k}\bar{z}|^{M+\alpha+2}}\,dA(z)\right)^{s}
	\lesssim\sum_{j,l,k}(1-|\z_{j,l}^{k}|^{2})^c\\
	&\asymp\sum_{j,l,k}\int_{\Delta(\z_{j,l}^{k},\delta)}(1-|z|^{2})^{c-2}\,dA(z)
	\lesssim\int_\D(1-|z|^{2})^{c-2}\,dA(z)<\infty,
	\end{split}
	\end{equation*}
and hence \eqref{fubinixxx} is valid.

Finally, by elementary calculations based on \cite[Lemma~2.1]{PelSum14} and \eqref{Eq:Dd-characterization}, we deduce
	$$
	(1-|\z_{j,l}^{k}|^{2})^{sM+\alpha s+2s-\frac{s}{s'}}\widehat{U}(\rho_{j})^{-\frac{s}{s'}}
	\lesssim
	V(\z_{j,l}^{k})^{\frac{s}{q}}(1-|\z_{j,l}^{k}|^{2})^{\frac{2s}{q}+sM-\frac{s}{p}}\widehat{\omega}(\z_{j,l}^{k})^{-\frac{s}{p}}
	$$
for all $j\in\mathbb{N}\cup\{0\},~l=0,1,\ldots,K^{j+3}-1,~k=1,2,\ldots,M^{2}$, and therefore
	\begin{equation*}
  \begin{split}
	\left\|\left(F_{f,n}\right)_r\right\|_{A_{W}^{s}}^s
	\lesssim
	\sum_{j,l,k}\frac{V(\z_{j,l}^{k})^{\frac{s}{q}}(1-|\z_{j,l}^{k}|^{2})^{\frac{2s}{q}+sM-\frac{s}{p}}}{\widehat{\omega}(\z_{j,l}^{k})^{\frac{s}{p}}}
\left|\int_{\mathbb{D}}\left(F_{f,n}\right)_r(z)\frac{\mu_{\nu}^{n}(z)}{(1-\z_{j,l}^{k}\bar{z})^{M+\alpha+2}}\,dA(z)\right|^{s}.
\end{split}
\end{equation*}
By applying Fatou's lemma and combining this estimate with \eqref{Ie:dualrelation1}, we obtain
	$$
	\left\|\frac{d^n}{d(\cdot)^n}((\cdot)^nf(\cdot))\right\|_{A_{W}^{s}}
	\le\liminf_{r\to1^{-}}\left\|\left(F_{f,n}\right)_r\right\|_{A_{W}^{s}}
	\le\limsup_{r\to1^{-}}\left\|\left(F_{f,n}\right)_r\right\|_{A_{W}^{s}}
	\lesssim\|h_{f}^{\nu}\|_{A_{\omega}^{p}\to A_{\nu}^{q}}.
	$$
Lemma~\ref{Lemma:derivative} now shows that the operator norm obeys the claimed lower bound for each fixed $n\ge \max\{N_1,N_2\}$.
\hfill$\Box$

\medskip

We finish the proof of the theorem by noting that the number $N=N(\om,\nu,p,q)\in\N$ appearing in the statement equals to $\max\{N_1,N_2\}$.

\section{Proof of Theorem~\ref{weak factorization-t1}}

The argument applied in the proof of Theorem~\ref{theorem} allows us to characterize the boudedness of the small Hankel operator in a certain three-weight case. This result is the tool that we will use to obtain a weak factorization for the weighted Bergman space induced by a radial weight in $\DDD$. To give the exact statement, we write $h^\eta_f=h^\alpha_f$ and $A^1_{\eta\log}=A^1_{\alpha\log}$ when $\eta(z)=(\alpha+1)(1-|z|)^\alpha$ for $-1<\alpha<\infty$. Minor modifications in the proof of Theorem~\ref{theorem} yield the following result, the proof of which is left to the interested readers. Observe that if $\left(\frac{\widetilde{\nu}}{\widetilde{\om}}\right)^{\frac{p}{p-q}}\widetilde{\om}\in\DDD$, then one may choose $n=0$ by the Littlewood-Paley formulas.

\begin{theorem}\label{theorem-three weights-1}
Let $1<q<p<\infty$ and $\om,\nu\in\DDD$. Then there exist $\alpha_{0}=\alpha_{0}(\om,\nu,p,q)>0$ and $N=N(\om,\nu,p,q)\in\N$ such that, for $f\in A_{\alpha\log}^{1}$, $h^\alpha_f: A^p_{\om}\rightarrow A^q_{\nu}$ is bounded if and only if
	\begin{equation}\label{appendix-eq} \int_\D\left|f^{(n)}(z)\right|^{\frac{pq}{p-q}}(1-|z|)^{n\frac{pq}{p-q}}\left(\frac{\widehat{\nu}(z)}{\whw(z)}\right)^{\frac{p}{p-q}}\frac{\widehat{\om}(z)}{1-|z|}\,dA(z)<\infty
	\end{equation}
for each $\alpha\geq\alpha_{0}$ and for some (equivalently for all) $n\ge N$. Moreover,
	$$
	\|h^\alpha_f\|^{\frac{pq}{p-q}}_{A^p_\om\rightarrow A^{q}_{\nu}}
	\asymp
	\int_\D\left|f^{(n)}(z)\right|^{\frac{pq}{p-q}}(1-|z|)^{n\frac{pq}{p-q}}\left(\frac{\widehat{\nu}(z)}{\whw(z)}\right)^{\frac{p}{p-q}}\frac{\widehat{\om}(z)}{1-|z|}\,dA(z)+\sum_{j=0}^{n-1}|f^{(j)}(0)|^{\frac{pq}{p-q}}
	$$
for each fixed $\alpha\geq\alpha_{0}$ and $n\ge N$.
\end{theorem}

The proof of Theorem~\ref{weak factorization-t1} also relies on the so-called Hankel type bilinear form. Following \cite{Pau-Zhao}, for $g,h\in H^{\infty}$, we define the Hankel type bilinear form $T_{f}^{\alpha}$, associated with a number $-1<\alpha<\infty$ and the symbol $f\in \mathcal{H}(\mathbb{D})$, by
	$$
	T_{f}^{\alpha}(g,h)=\langle gh,f\rangle_{A_{\alpha}^{2}}=\int_{\mathbb{D}}g(z)h(z)\overline{f(z)}dA_{\alpha}(z).
	$$
Since $H^\infty$ is dense in each weighted Bergman space induced by a radial weight, $T_{f}^{\alpha}$ is densely defined on $A_{\omega}^{p_{1}}\times A_{\nu}^{p_{2}}$ for each $0<p_{1},p_{2}<\infty$ and $\omega,\nu\in\mathcal{D}$. We say that $T_{f}^{\alpha}:A_{\omega}^{p_{1}}\times A_{\nu}^{p_{2}}\to\mathbb{C}$ is bounded if there exists an absolute constant $C>0$ such that
	$$
	|T_{f}^{\alpha}(g,h)|\leq C\|g\|_{A_{\omega}^{p_{1}}}\|h\|_{A_{\nu}^{p_{2}}},\quad g\in A_{\omega}^{p_{1}}, \quad h\in A_{\nu}^{p_{2}}.
	$$
Moreover,
$$
\|T_{f}^{\alpha}\|_{A_{\omega}^{p_{1}}\times A_{\nu}^{p_{2}}\to \mathbb{C}}
=\sup_{\{g,h:~\|g\|_{A_{\omega}^{p_{1}}}=\|h\|_{A_{\nu}^{p_{2}}}=1\}}|T_{f}^{\alpha}(g,h)|.
$$

The following result shows a close relationship between the boundedness of Hankel type bilinear forms and small Hankel operators. This connection plays a crucial role in the proof of Theorem~\ref{weak factorization-t1}.

\begin{proposition}\label{weak factorization-pro2}
Let $1<p_{2}'<p_{1}<\infty$ and $\omega,\nu\in\mathcal{D}$. Then there exists an $\alpha_{0}=\alpha_{0}(\om,\nu,p_{1},p_{2})>0$ such that, for each $\alpha\ge\alpha_{0}$
and $f\in A_{\alpha\log}^{1}$, the following statements are equivalent:
\begin{itemize}
\item[(i)] The Hankel type bilinear form $T_{f}^{\alpha}:A_{\omega}^{p_{1}}\times A_{\nu}^{p_{2}}\to\mathbb{C}$ is bounded;
\item[(ii)] The small Hankel operator $h_{f}^{\alpha}: A_{\omega}^{p_1}\to A_{\lambda}^{p_{2}'}$ is bounded;
\item[(iii)] $f\in A_{\sigma}^{\frac{p_{1}p_{2}'}{p_{1}-p_{2}'}}$, where
	$$
	\lambda(z)=\left(\frac{(1-|z|^{2})^{\alpha}}{\widetilde{\nu}(z)}\right)^{p_{2}'}\widetilde{\nu}(z),\quad
	\sigma(z)=\left(\frac{\widetilde{\lambda}(z)}{\widetilde{\omega}(z)}\right)^{\frac{p_{1}}{p_{1}-p_{2}'}}\widetilde{\omega}(z),\quad z\in\mathbb{D}.
	$$
\end{itemize}
Moreover,
	$$
	\|T_{f}^{\alpha}\|_{A_{\omega}^{p_{1}}\times A_{\nu}^{p_{2}}\to\mathbb{C}}\asymp\|h_{f}^{\alpha}\|_{A_{\omega}^{p_{1}}\to A_{\lambda}^{p_{2}'}}
	\asymp\|f\|_{A_{\sigma}^{\frac{p_{1}p_{2}'}{p_{1}-p_{2}'}}},
	$$
for each fixed $\alpha\geq\alpha_{0}$.
\end{proposition}

\begin{proof}
It is easy to see that there exists an $\alpha_1=\alpha_1(\omega,\nu,p_{1},p_{2})>0$ such that $\lambda,\sigma\in\mathcal{R}$ for all $\alpha\ge\alpha_1$. Therefore Theorem~\ref{theorem-three weights-1} 
 shows that, there exists an $\alpha_2=\alpha_2(\omega,\nu,p_{1},p_{2})\geq\alpha_{1}$ such that $h_{f}^{\alpha}: A_{\omega}^{p_1}\to A_{\lambda}^{p_{2}'}$ is bounded if and only if $f\in A_{\sigma}^{\frac{p_{1}p_{2}'}{p_{1}-p_{2}'}}$ for each $\alpha\geq \alpha_{2}$, and
	$$
	\|h_{f}^{\alpha}\|_{A_{\omega}^{p_{1}}\to A_{\lambda}^{p_{2}'}}
	\asymp\|f\|_{A_{\sigma}^{\frac{p_{1}p_{2}'}{p_{1}-p_{2}'}}},
	$$
for each fixed $\alpha\geq \alpha_{2}$.
Thus (ii) and (iii) are equivalent.

We next verify the equivalence between (i) and (ii). Assume first (ii). Fubini's theorem yields
	\begin{align*}
	T_{f}^{\alpha}(g,h)
	&=\langle gh,f\rangle_{A_{\alpha}^{2}}=\langle h,\overline{g}f\rangle_{L_{\alpha}^{2}}\\
	&=\langle P_{\alpha}(h),\overline{g}f\rangle_{L_{\alpha}^{2}}
	=\langle h,P_{\alpha}(\overline{g}f)\rangle_{A_{\alpha}^{2}}
	=\langle h,h_{f}^{\alpha}(g)\rangle_{A_{\alpha}^{2}},\quad g,h\in H^{\infty}.
	\end{align*}
This together with H\"{o}lder's inequality, the assumption (ii) and \cite[Proposition~5]{PRS1} give
	$$
	|T_{f}^{\alpha}(g,h)|
	=\left|\langle h,h_{f}^{\alpha}(g)\rangle_{A_{\alpha}^{2}}\right|\leq\|h_{f}^{\alpha}(g)\|_{A_{\lambda}^{p_{2}'}}\|h\|_{A_{\widetilde{\nu}}^{p_{2}}}
	\lesssim\|h_{f}^{\alpha}\|_{A_{\omega}^{p_{1}}\to A_{\lambda}^{p_{2}'}}\|g\|_{A_{\omega}^{p_{1}}}\|h\|_{A_{\nu}^{p_{2}}},
	\quad g,h\in H^{\infty},
	$$
and hence
 $T_{f}^{\alpha}:A_{\omega}^{p_{1}}\times A_{\nu}^{p_{2}}\to\mathbb{C}$ is bounded and
$\|T_{f}^{\alpha}\|_{A_{\omega}^{p_{1}}\times A_{\nu}^{p_{2}}\to \mathbb{C}}\lesssim\|h_{f}^{\alpha}\|_{A_{\omega}^{p_{1}}\to A_{\lambda}^{p_{2}'}}$.

Conversely, if (i) holds, then
	\begin{equation}\label{new-weak-eq1}
	\left|\langle h,h_{f}^{\alpha}(g)\rangle_{A_{\alpha}^{2}}\right|
	=|T_{f}^{\alpha}(g,h)|
	\le\|T_{f}^{\alpha}\|_{A_{\omega}^{p_{1}}\times A_{\nu}^{p_{2}}\to \mathbb{C}}\|g\|_{A_{\omega}^{p_{1}}}\|h\|_{A_{\nu}^{p_{2}}},\quad g,h\in H^{\infty}.
	\end{equation}
Since $\lambda\in\mathcal{R}$ for $\alpha\ge\alpha_1$, \cite[Theorem 3 and Theorem 6]{PR2014} show that the dual of $A_{\widetilde{\nu}}^{p_{2}}$ can be identified with $A_{\lambda}^{p_{2}'}$ via the $A_{\alpha}^{2}$-pairing.
This together with \eqref{new-weak-eq1} and \cite[Proposition 5]{PRS1} shows that $h_{f}^{\alpha}: A_{\omega}^{p_{1}}\to A_{\lambda}^{p_{2}'}$ is bounded and
$\|h_{f}^{\alpha}\|_{A_{\omega}^{p_{1}}\to A_{\lambda}^{p_{2}'}}\lesssim\|T_{f}^{\alpha}\|_{A_{\omega}^{p_{1}}\times A_{\nu}^{p_{2}}\to \mathbb{C}}$.

We finish the proof by noting that the number $\alpha_{0}=\alpha_{0}(\om,\nu,p_{1},p_{2})>0$ appearing in the statement equals to $\alpha_{2}$.
\end{proof}

With these preparations we are ready for the proof of Theorem~\ref{weak factorization-t1}.
\medskip

\noindent\emph{Proof of Theorem \ref{weak factorization-t1}.} On one hand, each $b\in A_{\omega}^{p_{1}}\odot A_{\nu}^{p_{2}}$ can be decomposed as
	$$
	b=\sum_{k}\varphi_{k}\psi_{k},\quad\{\varphi_{k}\}\subset A_{\omega}^{p_{1}},\quad\{\psi_{k}\}\subset A_{\nu}^{p_{2}},
	$$
from which Minkowski's inequality, H\"{o}lder's inequality and \cite[Proposition~5]{PRS1} yield
	$$
	\|b\|_{A_{\eta}^{q}}=\bigg(\int_{\mathbb{D}}\bigg|\sum_{k}\varphi_{k}\psi_{k}\bigg|^{q}\eta\,dA\bigg)^{\frac{1}{q}}
	\le\sum_{k}\bigg(\int_{\mathbb{D}}\big|\varphi_{k}\psi_{k}\big|^{q}\eta\,dA\bigg)^{\frac{1}{q}}
	\lesssim \sum_{k}\|\varphi_{k}\|_{A_{\omega}^{p_{1}}}\|\psi_{k}\|_{A_{\nu}^{p_{2}}}.
	$$
It follows that
	\begin{equation}\label{weak-factorization-111111}
	A_{\omega}^{p_{1}}\odot A_{\nu}^{p_{2}}\subset A_{\eta}^{q}
	\end{equation}
and
	\begin{equation}\label{weak-factorization-111111-}
	\|b\|_{A_{\eta}^{q}}\lesssim \|b\|_{A_{\omega}^{p_{1}}\odot A_{\nu}^{p_{2}}}.
	\end{equation}

On the other hand, if $b\in A_{\eta}^{q}$, then \cite[Theorem 2]{PRS-2021} implies that there exist an $M_{0}=M_{0}(q,\eta)>0$ and a sequence $\{\gamma(b)_{j,l}^{k}\}\in\ell^{q}$ such that
	\begin{equation}\label{weak factorization-eq5}
	b(z)=\sum_{j,l,k}\gamma(b)_{j,l}^{k}\frac{(1-|\zeta_{j,l}^{k}|^{2})^{M-\frac{1}{q}}\widehat{\eta}(\rho_{j})^{-\frac{1}{q}}}
	{(1-\overline{\zeta_{j,l}^{k}}z)^{M}},\quad z\in\mathbb{D},
	\end{equation}
for all $M\ge M_{0}$, where $\zeta_{j,l}^{k}$ and $\rho_{j}$ are the same as those appearing in the proof of Theorem~\ref{theorem}.
Rewrite $b=\sum_{j,l,k}\varphi_{j,l}^{k}\psi_{j,l}^{k}$,
where
	$$
	\varphi_{j,l}^{k}(z)
	=\frac{\gamma(b)_{j,l}^{k}\widehat{\eta}(\rho_{j})^{-\frac{1}{q}}}{\widehat{\nu}(\zeta_{j,l}^{k})^{-\frac{1}{p_{2}}}
	\widehat{\omega}(\zeta_{j,l}^{k})^{-\frac{1}{p_{1}}}}
	\frac{(1-|\zeta_{j,l}^{k}|^{2})^{\frac{M}{2}-\frac{1}{p_{1}}}\widehat{\omega}(\zeta_{j,l}^{k})^{-\frac{1}{p_{1}}}}
	{(1-\overline{\zeta_{j,l}^{k}}z)^{\frac{M}{2}}},
	$$
and
	$$
	\psi_{j,l}^{k}(z)=\frac{(1-|\zeta_{j,l}^{k}|^{2})^{\frac{M}{2}-\frac{1}{p_{2}}}\widehat{\nu}(\zeta_{j,l}^{k})^{-\frac{1}{p_{2}}}}
	{(1-\overline{\zeta_{j,l}^{k}}z)^{\frac{M}{2}}}.
	$$
Since $\widetilde{\eta}$ is a regular weight, we have
\begin{equation}\label{weak factorization-eq30}
	\widehat{\eta}(\rho_{j})^{-\frac{1}{q}}
	\asymp\widehat{\widetilde{\eta}}(\rho_{j})^{-\frac{1}{q}}
	\asymp\widehat{\widetilde{\eta}}(\zeta_{j,l}^{k})^{-\frac{1}{q}}
	\asymp\widetilde{\eta}(\zeta_{j,l}^{k})^{-\frac{1}{q}}(1-|\zeta_{j,l}^{k}|)^{-\frac{1}{q}}
	\asymp\widehat{\nu}(\zeta_{j,l}^{k})^{-\frac{1}{p_{2}}}\widehat{\omega}(\zeta_{j,l}^{k})^{-\frac{1}{p_{1}}}
	\end{equation}
for all $j, l, k$. Then \cite[Lemma~2.1]{PelSum14}, \eqref{weak factorization-eq30}, and the fact that $\{\gamma(b)_{j,l}^{k}\}\in\ell^{q}$ show that there exists an $M_{1}=M_{1}(p_{1},p_{2},\omega,\nu)\geq M_{0}$ such that $\varphi_{j,l}^{k}\in {A_{\omega}^{p_{1}}}$ and $\psi_{j,l}^{k}\in{A_{\nu}^{p_{2}}}$ for all $j,l,k$ and $M\ge M_{1}$.

To complete the proof, it remains to show that $\|b\|_{A_{\omega}^{p_{1}}\odot A_{\nu}^{p_{2}}}\lesssim\|b\|_{A_{\eta}^{q}}$ for $b\in A^q_\eta$. Note that the hypotheses $1<q,p_{1},p_{2}<\infty$ and $\frac{1}{p_{1}}+\frac{1}{p_{2}}=\frac{1}{q}$
imply
	\begin{equation}\label{weak-factorization----}
	1<p_{2}'<p_{1}<\infty.
	\end{equation}
Let $\alpha\in(-1,\infty)$ to be fixed later, and let
	$$
	\sigma(z)=\left(\frac{(1-|z|^{2})^{\alpha}}{\widetilde{\eta}(z)}\right)^{q'}\widetilde{\eta}(z),\quad z\in\mathbb{D}.
	$$
We claim that for each $F\in (A_{\omega}^{p_{1}}\odot A_{\nu}^{p_{2}})^{\star}$ there exists an $f_{F}\in A_{\sigma}^{q'}$ such that $F(b)=\langle b,f_{F}\rangle_{A_{\alpha}^{2}}$ for all $b\in H^\infty$, and $\|f_{F}\|_{A_{\sigma}^{q'}}\lesssim\|F\|$. The proof of this fact is postponed for a moment. It together with H\"{o}lder's inequality and \cite[Proposition 5]{PRS1} yields
	\begin{align*}
 \|b\|_{A_{\omega}^{p_{1}}\odot A_{\nu}^{p_{2}}}&=\sup_{\left\{F\in (A_{\omega}^{p_{1}}\odot A_{\nu}^{p_{2}})^{\star}:~ \|F\|=1\right\}}|F(b)|
	=\sup_{\left\{F\in (A_{\omega}^{p_{1}}\odot A_{\nu}^{p_{2}})^{\star}:~\|F\|=1\right\}}|\langle b,f_{F}\rangle_{A_{\alpha}^{2}}|\\
  &\lesssim\sup_{\{F\in (A_{\omega}^{p_{1}}\odot A_{\nu}^{p_{2}})^{\star}:~\|F\|=1\}}\|f_{F}\|_{A_{\sigma}^{q'}}\|b\|_{A_{\eta}^{q}}\\
  &\lesssim\|b\|_{A_{\eta}^{q}},\quad b\in H^\infty,
	\end{align*}
which completes the proof through the density argument.

It remains to establish the fact claimed above. For each $F\in(A_{\omega}^{p_{1}}\odot A_{\nu}^{p_{2}})^{\star}$, \cite[Proposition 5]{PRS1} yields
	$$
	|F(b)|
	\le\|F\|\|(1\cdot b)\|_{A_{\omega}^{p_{1}}\odot A_{\nu}^{p_{2}}}
	\le\|F\|\|1\|_{A_{\omega}^{p_{1}}}\|b\|_{A_{\nu}^{p_{2}}}
	\asymp\|F\|\|b\|_{A_{\widetilde{\nu}}^{p_{2}}}, \quad b\in A_{\widetilde{\nu}}^{p_{2}},
	$$
which means that $F$ is a bounded linear functional also on $A_{\widetilde{\nu}}^{p_{2}}$. Let now
	$$
	\lambda(z)=\bigg(\frac{(1-|z|^{2})^{\alpha}}{\widetilde{\nu}(z)}\bigg)^{p_{2}'}\widetilde{\nu}(z),\quad z\in\mathbb{D}.
	$$
Note that there exists an $\alpha_{0}=\alpha_{0}(\nu,p_{2})>0$ such that $\lambda$ is a regular weight for each $\alpha\geq \alpha_{0}$. Then \cite[Theorems 3 and~6]{PR2014} gives $(A^{p_{2}}_{\widetilde{\nu}})^{\star}\simeq A_{\lambda}^{p_{2}'}$ via the $A_{\alpha}^{2}$-pairing. Therefore there exists an $f=f_{F}\in A_{\lambda}^{p_{2}'}\subset A_{\alpha\log}^{1}$ such that
$F(b)=\langle b,f\rangle_{A_{\alpha}^{2}}$ for all $b\in A_{\widetilde{\nu}}^{p_{2}}$. Observe that this does not quite give us what we want because $A_{\lambda}^{p_{2}'}\not\subset A_{\sigma}^{q'}$ due to $q'>p_2'$ as is seen by applying \cite[Theorem~2]{LiuRattya}: after routine calculations $\sigma(\Delta(z,r))\lesssim\left(\lambda(\Delta(z,r))\right)^\frac{q'}{p_2'}$ reduces to the inequality $p_2'>q'$. However, it allows us to deduce
	$$
	|T_{f}^{\alpha}(g,h)|
	=|\langle gh,f\rangle_{A_{\alpha}^{2}}|
	=|F(gh)|\le\|F\|\|gh\|_{A_{\omega}^{p_{1}}\odot A_{\nu}^{p_{2}}}
	\le\|F\|\|g\|_{A_{\omega}^{p_{1}}}\|h\|_{A_{\nu}^{p_{2}}},\quad g,h\in H^{\infty},
	$$
which further implies that the Hankel type bilinear form $T_{f}^{\alpha}:A_{\omega}^{p_{1}}\times A_{\nu}^{p_{2}}\to\mathbb{C}$ is bounded with
$\|T_{f}^{\alpha}\|_{A_{\omega}^{p_{1}}\times A_{\nu}^{p_{2}}\to\mathbb{C}}\leq\|F\|$ for each $\alpha\geq \alpha_{0}$. This together with Proposition~\ref{weak factorization-pro2} and \eqref{weak-factorization----} imply the existence of an $\alpha_{1}=\alpha_{1}(\omega,\nu,p_{1},p_{2})\ge\alpha_{0}$ such that $f\in A_{\sigma}^{q'}$ with
	\[
	\|f\|_{A_{\sigma}^{q'}}\asymp\|T_{f}^{\alpha}\|_{A_{\omega}^{p_{1}}\times A_{\nu}^{p_{2}}\to \mathbb{C}}\leq\|F\|
	\]
for each $\alpha\geq \alpha_{1}$. Thus $F(b)=\langle b,f\rangle_{A_{\alpha}^{2}}$ for all $b\in H^\infty$, and $\|f\|_{A_{\sigma}^{q'}}\lesssim\|F\|$.
\hfill$\Box$

\section{Proof of Theorem~\ref{theorem1}}

In the proof of Theorem~\ref{theorem} we used the BLT-theorem \cite[Theorem~I.7]{ReedSimon} and the density of $H^\infty$ in $A^p_\om$ to prove the sufficiency. In the other theorems, excepting Theorem~\ref{theorem2}, $h^\nu_f$ is operating from $A^p_\om$ to $A^q_\nu$, but either $p$ or $q$ might be strictly less than 1, and then the space induced is not normed and then neither Banach. Therefore we cannot appeal to the BLT-theorem as such. However, $A^p_\om$ is a quasinormed space equipped with $\|\cdot\|_{A_{\omega}^{p}}$ in the case $0<p<1$, and the linear space $A^p_\om$ is a complete metric space with the metric defined by $d(f,g)=\|f-g\|_{A_{\omega}^{p}}^p$, see, for instance, \cite[Appendix~A]{Pav}. Now a careful inspection of the proof of the BLT-theorem given in \cite{ReedSimon}, with minor modifications, shows that the statement is valid for $h^\nu_f:A^p_\om\to A^q_\nu$ on the full range of $0<p,q<\infty$ and $\omega,\nu\in\mathcal{D}$. Hence we may use the BLT-theorem in any case in our setting of Bergman spaces, and we will still call it the density argument as usual. This simplifies certain steps in our proofs because we do not have to worry much about the convergence of integrals while using Fubini's theorem. However, it is worth underlining here that one may avoid the use of the BLT-theorem by directly verifying the validity of the formulas where the order of integration is changed. This can be done in each case by relying on the corresponding hypothesis on the symbol $f$ and performing somewhat tedious calculations.

If $0<p\le q<\infty$ and $\om\in\DDD$, then \cite[Theorem~2]{LiuRattya} states that there exists an $r_0=r_0(\om)\in(0,1)$ such that a positive Borel measure $\mu$ on $\D$ is a $q$-Carleson measure for $A^p_\om$ if and only if
	\begin{equation}\label{tvtvt}
	\sup _{z \in \mathbb{D}} \frac{\mu(\Delta(z, r))}{(\omega(\Delta(z, r)))^{\frac{q}{p}}}<\infty
	\end{equation}
for some (equivalently for all) $r\in(r_0,1)$. We will use this result to prove the sufficiency in Theorem~\ref{theorem1}.

\medskip

\noindent\emph{Proof of sufficiency}. In the proof of the sufficiency in Theorem~\ref{theorem} we established the estimate
	\begin{align*}
	\|h^\nu_f(g)\|^q_{A^q_\nu}\lesssim\int_\D |g(\z)|^q\left|\frac{d^n}{d\z^n}(\z^nf(\z))\right|^q(1-|\z|)^{nq}\widetilde{\nu}(\z)\,dA(\z)=\int_\D |g(\z)|^q\,d\eta(\z),\quad g\in H^\infty,
	\end{align*}
where
	$$
	d\eta(\z)=d\eta^n_{f,\nu,q}(\z)=\left|\frac{d^n}{d\z^n}(\z^nf(\z))\right|^q(1-|\z|)^{nq}\widetilde{\nu}(\z)\,dA(\z)
	$$
and $n\in\N\cup\{0\}$ is arbitrarily fixed. Therefore to show that \eqref{Eq:theorem2} is a sufficient condition for $h^\nu_f:A^p_\om\rightarrow A^{q}_{\nu}$ to be bounded, it suffices to show that $\eta$ is a $q$-Carleson measure for $A^p_\om$, provided $n$ is sufficiently large. By \cite[Theorem~2]{LiuRattya} this is equivalent to \eqref{tvtvt} with $\mu=\eta$. To see this, we first observe that since $\om\in\DDD$, \cite[Lemma~2.1]{PelSum14} and \eqref{Eq:Dd-characterization} show that there exists an $r_0=r_0(\om)\in(0,1)$ such that $\om(\Delta(z,r))\asymp\widehat{\om}(z)(1-|z|)$ for all $z\in\D$, provided $r\in[r_0,1)$ is fixed. This asymptotic equality together with the assumption \eqref{Eq:theorem2} and \cite[Lemma~2.1]{PelSum14} yields
	\begin{align*}
	\frac{\eta(\Delta(z,r))}{\om(\Delta(z,r))^{\frac qp}}
	&\asymp\frac{\int_{\Delta(z,r)}\left|\frac{d^n}{d\z^n}(\z^nf(\z))\right|^q(1-|\z|)^{nq}\widetilde{\nu}(\z)\,dA(\z)}
	{\left(\widehat{\om}(z)(1-|z|)\right)^{\frac qp}}\\
	&\lesssim\frac{\int_{\Delta(z,r)}\widehat{\om}(\z)^{\frac qp}(1-|\z|)^{\frac{q}{p}-2}\,dA(\z)}
{\left(\widehat{\om}(z)(1-|z|)\right)^{\frac qp}}\asymp1,\quad z\in\D.
	\end{align*}
Then the density argument shows that $h^\nu_f:A^p_\om\to A^q_\nu$ is bounded, and
	$$
	\|h^\nu_{f}\|_{A^p_\om\rightarrow A^q_\nu}
	\lesssim\sup_{\z\in\D}\left|\frac{d^n}{d\z^n}(\z^nf(\z))\right|(1-|\z|)^n
	\frac{\left(\widehat{\nu}(\z)(1-|\z|)\right)^{\frac{1}{q}}}{\left(\whw(\z)(1-|\z|)\right)^{\frac{1}{p}}}.
	$$
An application of Lemma~\ref{Lemma:derivative2} completes the proof, provided $n\in\N$ is sufficiently large, say $n\ge N_1=N_1(\om,\nu,p,q)$. \hfill$\Box$

\medskip


For the necessity we will need the following kernel estimate, the proof of which can be found in~\cite{PR2017}. This estimate will also play an important role in the proof of Theorem~\ref{theorem3}.

\begin{lemma}\label{weightedkernel}
Let $2\le p<\infty$, $k\in\N\cup\{0\}$ and $\nu\in\DD$, and let $\om$ be a radial weight. Then
	\begin{equation*}
  \int_{\D}|(1-\overline{z}\z)^k B^\nu_z(\z)|^p\om(\z)\,dA(\z)
	\lesssim\int_0^{|z|}\frac{\whw(t)}{\widehat{\nu}(t)^p(1-t)^{p(1-k)}}\,dt+1,\quad z\in\D.
	\end{equation*}
\end{lemma}

\medskip

We can now prove the necessity.

\medskip

\noindent\emph{Proof of necessity}. Assume that $h^\nu_f:A^p_\om\rightarrow A^{q}_{\nu}$ is bounded. Since $f\in A^1_{\nu_{\log}}$ by the hypothesis, \cite[Theorem~1]{PR2014} yields
	\begin{equation*}
	\begin{split}
	\int_\D|f(\zeta)||g(\zeta)|\left(\int_\D|B^\nu_z(\zeta)||h(z)|\nu(z)\,dA(z)\right)\nu(\zeta)\,dA(\zeta)
	\lesssim\|g\|_{H^\infty}\|h\|_{H^\infty}\|f\|_{A^1_{\nu_{\log}}}<\infty
	\end{split}
	\end{equation*}
for all $g,h\in H^\infty$. Therefore Fubini's theorem and the boundedness of $h^\nu_f:A^p_\om\rightarrow A^{q}_{\nu}$ imply
	\begin{equation*}
	\begin{split}
  \left|\int_{\D}f(\z)\overline{g(\z)}\overline{h(\z)}\nu(\z)\,dA(\z)\right|
	&=\left|\left<h^\nu_{f}(g), h\right>_{A^2_\nu}\right|
	\le\|h^\nu_f(g)\|_{A^{q}_{\nu}}\|h\|_{A^{q'}_\nu}\\
	&\le\|h^\nu_f\|_{A^p_\om\rightarrow A^{q}_{\nu}}\|g\|_{A^p_\om}\|h\|_{A^{q'}_\nu},\quad g,h\in H^\infty.
	\end{split}
	\end{equation*}
Let $k\in\N$ to be fixed later. The above inequality with the choices $g(\z)=g_{z,k}(\z)=(1-\overline{z}\z)^{-k}$ and $h(\z)=h_{z,k,n,\nu}(\z)=(1-\overline{z}\z)^{k}B^{\mu_\nu^n}_z(\z)$, together with Lemma~\ref{nthreproducing} gives
	\begin{equation}\label{testfg}
	\left|\frac{d^n}{dz^n}(z^nf(z))\right|
	\le\|h^\nu_f\|_{A^p_\om\rightarrow A^{q}_{\nu}} \|g_{z,k}\|_{A^p_\om}\|h_{z,k,n,\nu}\|_{A^{q'}_\nu},\quad z\in\D.
	\end{equation}
By \cite[Lemma~2.1]{PelSum14} we may fix $k=k(\om,p)\in\N$ large enough such that
	\begin{equation}\label{testf}
  \|g_{z,k}\|^p_{A^p_\om}
	\asymp\frac{\whw(z)}{(1-|z|)^{pk-1}},\quad z\in\D.
	\end{equation}
Further, we claim that there exists an $N_2=N_2(\om,\nu,p,q)\in\N$ such that, for all $n\ge N_2$, we have
	\begin{equation}\label{testg}
	\|h_{z,k,n,\nu}\|^{q'}_{A^{q'}_\nu}
	\lesssim\frac{1}{\widehat{\nu}(z)^{q'-1}(1-|z|)^{q'(n+1-k)-1}},\quad z\in\D,
	\end{equation}
the proof of which is postponed for a moment. By applying \eqref{testf} and \eqref{testg} to \eqref{testfg}, we obtain
	\begin{equation*}
	\left|\frac{d^n}{dz^n}(z^nf(z))\right|
	\lesssim\frac1{(1-|z|)^n}\frac{\left(\whw(z)(1-|z|)\right)^{\frac{1}{p}}}{\left(\widehat{\nu}(z)(1-|z|)\right)^{\frac{1}{q}}},\quad z\in\D.
	\end{equation*}
By choosing $N_3\ge N_2$ sufficiently large so that Lemma~\ref{Lemma:derivative2} can be applied, we finally deduce the desired lower bound for the operator norm, provided $n\ge N_3$.

To complete the proof it remains to establish \eqref{testg}. Let first $1<q\leq 2$, and recall that $\mu_\nu^n\in\DDD$ with $\widehat{\mu_\nu^n}(z)\asymp\widehat{\nu}(z)(1-|z|)^n$ for all $z\in\D$ by \eqref{Ieq:newweight}. Let $n\ge k$, where $k=k(\om,p)\in\N$ is that of \eqref{testf}, and apply Lemma~\ref{weightedkernel} to deduce
	\begin{equation*}
	\begin{split}
	\|h_{z,k,n,\nu}\|^{q'}_{A^{q'}_\nu}
	&\lesssim\int_0^{|z|}\frac{\widehat{\nu}(t)}{\widehat{\mu_\nu^n}(t)^{q'}(1-t)^{q'(1-k)}}\,dt+1
	\asymp\int_0^{|z|}	\frac{dt}{\widehat{\nu}(t)^{q'-1}(1-t)^{q'(n+1-k)}}+1\\
	&\lesssim\frac{1}{\widehat{\nu}(z)^{q'-1}(1-|z|)^{q'(n+1-k)-1}},\quad z\in\D,
	\end{split}	
	\end{equation*}
and thus \eqref{testg} is valid in this case. Let now $2<q<\infty$. By \cite[Lemma~2.1]{PelSum14} we may fix $j=j(\nu,q)\in\N$ large enough such that
	$$
	\int_\D\frac{\nu(\z)}{\left|1-\overline{z}\z\right|^{\frac{2q'j}{2-q'}}}\,dA(\z)
	\lesssim\frac{\widehat{\nu}(z)}{(1-|z|)^{\frac{2q'j}{2-q'}-1}},\quad z\in\D.
	$$
Let now $N_4=N_4(\om,\nu,p,q)=k+j$, and let $n\ge N_4$. Then H\"{o}lder's inequality, Lemma~\ref{weightedkernel} and \eqref{Ieq:newweight} imply
	\begin{align*}
  \|h_{z,k,n,\nu}\|^{q'}_{A^{q'}_\nu}
	&=\int_\D\left|(1-\overline{z}\z)^{k+j}B^{\mu_\nu^n}_z(\z)\right|^{q'}\frac{\nu(\z)}{|1-\overline{z}\z|^{ q'j}}\,dA(\z)\\
  &\le\left(\int_\D\left|(1-\overline{z}\z)^{k+j}B^{\mu_\nu^n}_z(\z)\right|^{2}\nu(\z)\,dA(\z)\right)^{\frac{q'}{2}}
	\cdot\left(\int_\D\frac{\nu(\z)}{\left|1-\overline{z}\z\right|^{\frac{2q'j}{2-q'}}}\,dA(\z)\right)^{\frac{2-q'}{2}}\\
  &\lesssim\left(\int_0^{|z|}\frac{\widehat{\nu}(t)}{\widehat{\mu_\nu^n}(t)^2(1-t)^{2(1-k-j)}}\,dt+1\right)^{\frac{q'}{2}}\cdot \left(\frac{\widehat{\nu}(z)}{(1-|z|)^{\frac{2q'j}{2-q'}-1}}\right)^{\frac{2-q'}{2}} \\
  &\asymp\left(\int_0^{|z|}\frac{dt}{\widehat{\nu}(t)(1-t)^{2(n+1-k-j)}}+1\right)^{\frac{q'}{2}}\cdot \frac{\widehat{\nu}(z)^{\frac{2-q'}{2}}}{(1-|z|)^{q' j-1+\frac{q'}{2}}}\\
  &\lesssim\left(\frac{1}{\widehat{\nu}(z)(1-|z|)^{2(n-k-j)+1}}\right)^{\frac{q'}{2}}\cdot \frac{\widehat{\nu}(z)^{\frac{2-q'}{2}}}{(1-|z|)^{q' j-1+\frac{q'}{2}}}\\
  &=\frac{1}{\widehat{\nu}(z)^{q'-1}(1-|z|)^{q'(n+1-k)-1}},\quad z\in\D,
  \end{align*}
and thus \eqref{testg} is valid in this case also. This completes the proof of the necessity. \hfill$\Box$

\medskip

We finish the proof of the theorem by noting that the number $N=N(\om,\nu,p,q)\in\N$ appearing in the statement obviously equals to $\max\{N_1,N_3,N_4\}$.

\section{Proof of Theorem~\ref{theorem2}}

We begin with showing that \eqref{Eq:theorem3} is a sufficient condition for $h_{f}^{\nu}:A^{p}_{\omega}\to A^{1}_{\nu}$ to be bounded.

\medskip

\noindent\emph{Proof of sufficiency}. First we notice that $(A^1_{\nu})^\star\simeq\B$ via the $A^2_\nu$-pairing
	$$
	\langle f,g\rangle_{A^2_\nu}=\lim_{r\to1^-}\int_\D f_r(z)\overline{g(z)}\nu(z)\,dA(z)
	$$
by \cite[Theorem~3]{PelRat2020} because $\nu\in\DDD$ by the hypothesis. We next aim for boosting the order of the derivative of $f$ in the quantity $\langle h_{f}^{\nu}(g),h\rangle_{A^{2}_{\nu}}$ with $g\in H^\infty$ and $h\in\B$. First observe that, by \cite[Theorem 1]{PR2014}, H\"older's inequality, the hypothesis \eqref{Eq:theorem3} and Lemma~\ref{Lemma:derivative}, there exists an $N_1=N_1(\om,\nu,p)\in\N$ such that
	\begin{equation*}
	\begin{split}
  &\int_\D\left|\frac{d^n}{d\z^n}(\z^nf(\z))\right||g(\z)|\mu^n_\nu(\z)
	\left(\int_\D|h(z)||B^\nu_\z(z)|\nu(z)\,dA(z)\right)dA(\z)\\
  &\lesssim\|g\|_{H^\infty}\|h\|_{\B}\int_\D\left|\frac{d^n}{d\z^n}(\z^nf(\z))\right|\mu^n_\nu(\z)
	\left(\int_\D|B^\nu_\z(z)|\log\frac{e}{1-|z|}\nu(z)\,dA(z)\right)\,dA(\z)\\
	&\lesssim\|g\|_{H^\infty}\|h\|_{\B}\int_\D\left|\frac{d^n}{d\z^n}(\z^nf(\z))\right|\mu^n_\nu(\z)
	\left(\int_0^{|\z|}\frac{\int_t^1\log\frac{e}{1-s}\nu(s)\,ds}{\widehat{\nu}(t)(1-t)}\,dt+1\right)\,dA(\z)\\
	&\lesssim\|g\|_{H^\infty}\|h\|_{\B}\int_\D\left|\frac{d^n}{d\z^n}(\z^nf(\z))\right|\mu^n_\nu(\z)
	\left(\log\frac{e}{1-|\zeta|}\right)^2\,dA(\z)\\
	&\lesssim\|g\|_{H^{\infty}}\|h\|_{\B}\left(\int_{\mathbb{D}}\left|\frac{d^n}{d\z^n}(\z^nf(\z))\right|^{p'}(1-|\z|^{2})^{np'}
	\left(\frac{\widehat{\nu}(\z)}{\widehat{\omega}(\z)}\right)^{p'}\widetilde{\omega}(\z)\,dA(\z)\right)^{\frac{1}{p'}}\\
	&\quad\cdot\left(\int_{\mathbb{D}}\left(\log\frac{e}{1-|\z|}\right)^{2p}\widetilde{\omega}(\z)\,dA(\z)\right)^{\frac{1}{p}}
	<\infty,\quad g\in H^\infty,\quad h\in\B,
  \end{split}
	\end{equation*}
for all $n\ge N_1$. Therefore, by \eqref{Id: newrepresentation} and Fubini's theorem, we have
	\begin{equation}\label{predualofbloch}
	\begin{split}
  \langle h^\nu_f(g),h\rangle_{A^2_\nu}
	&=\int_\D\left(\int_\D\left(\frac{d^n}{d\z^n}(\z^nf(\z))\right)\overline{g(\z)}\overline{B^\nu_z(\z)}\mu^n_\nu(\z)\,dA(\z)\right)
	\overline{h(z)}\nu(z)\,dA(z)\\
  &=\int_\D\left(\frac{d^n}{d\z^n}(\z^nf(\z))\right)\overline{g(\z)}\mu^n_\nu(\z)
	\left(\overline{\int_\D h(z)\overline{B^\nu_\z(z)}\nu(z)\,dA(z)}\right)\,dA(\z)\\
  &=\int_\D\left(\frac{d^n}{d\z^n}(\z^nf(\z))\right)\overline{g(\z)}\overline{h(\z)}\mu^n_\nu(\z)\,dA(\z),\quad g\in H^\infty,\quad h\in\B.
  \end{split}
	\end{equation}
Then \cite[Proposition~5]{PRS1}, \eqref{Ieq:newweight1}, H\"older's inequality, and the hypothesis \eqref{Eq:theorem3} yield
	\begin{align*}
	|\langle h_{f}^{\nu}(g),h\rangle_{A^{2}_{\nu}}|
	&=\left|\int_{\mathbb{D}}\frac{d^n}{dz^n}(z^nf(z))\overline{g(z)h(z)}\mu_{\nu}^{n}(z)dA(z)\right|\nonumber\\
	&\lesssim\left(\int_{\mathbb{D}}\left|h(z)\right|^{p'}\left|\frac{d^n}{dz^n}(z^nf(z))\right|^{p'}(1-|z|^{2})^{np'}
	\left(\frac{\widehat{\nu}(z)}{\widehat{\omega}(z)}\right)^{p'}\widetilde{\omega}(z)dA(z)\right)^{\frac{1}{p'}}
	\|g\|_{A_{\widetilde{\omega}}^{p}}\nonumber\\
	&\lesssim\|h\|_{\mathcal{B}}\|g\|_{A^{p}_{\omega}}, \quad g\in H^\infty,\quad h\in\B.
	\end{align*}
It follows from \cite[Theorem~3]{PelRat2020} and the density argument that $h_{f}^{\nu}:A^{p}_{\omega}\to A^{1}_{\nu}$ is bounded and
	$$
	\|h_{f}^{\nu}\|_{A^{p}_{\omega}\to A^{1}_{\nu}}^{p'}
	\lesssim\sup_{h\in\B\setminus\{0\}}\frac{\int_{\mathbb{D}}\left|h(z)\right|^{p'}\left|\frac{d^n}{dz^n}(z^nf(z))\right|^{p'}(1-|z|^{2})^{np'}
	\left(\frac{\widehat{\nu}(z)}{\widehat{\omega}(z)}\right)^{p'}\widetilde{\omega}(z)dA(z)}{\|h\|_{\mathcal{B}}^{p'}}.
	$$
Lemma~\ref{lemma:bloch-measure} now completes the proof of the sufficiency, provided $n\ge N_2=N_2(\om,\nu,p)\ge N_1$. \hfill$\Box$

\medskip

\noindent\emph{Proof of necessity}. Assume that $h_{f}^{\nu}:A^{p}_{\omega}\to A^{1}_{\nu}$ is bounded, and write
	$$
	\tau(z)=\tau_{\om,\nu,p}(z)=\left(\frac{\widehat{\nu}(z)}{\widehat{\omega}(z)}\right)^{p'}\widetilde{\omega}(z),\quad z\in\mathbb{D},
	$$
for short. We will show that
	\begin{equation}\label{aimoftheorem3}
	\int_{\mathbb{D}}|h(z)|^{p'}\left|\frac{d^n}{dz^n}(z^nf(z))\right|^{p'}\tau_{[np']}(z)\,dA(z)
	\lesssim\|h_{f}^{\nu}\|_{A^{p}_{\om}\to A^{1}_{\nu}}^{p'}\|h\|_{\mathcal{B}}^{p'},\quad h\in\mathcal{B},
	\end{equation}
provided $n$ is sufficiently large, depending on $\om$, $\nu$ and $p$, and fixed. Clearly, $\B\subset A^1_\nu$ because $\nu\in\DDD$ by the hypothesis. Hence Fubini's theorem and the hypothesis $f\in A^1_{\nu_{\log}}$ give
	\begin{equation}\label{jdjdjdjdj}
	\begin{split}
	\langle h_{f}^{\nu}(g),h\rangle_{A^{2}_{\nu}}
	&=\langle P_{\nu}(f\overline{g}),h\rangle_{A^{2}_{\nu}}
	=\langle f\overline{g},P_{\nu}(h)\rangle_{L^{2}_{\nu}}
	=\langle f\overline{g},h\rangle_{L^{2}_{\nu}}\\
	&=\langle f\overline{h},g\rangle_{L^{2}_{\nu}}
	=\langle f\overline{h},P_{\mu^{n}_{\nu}}(g)\rangle_{L^{2}_{\nu}}
	=\langle S^{\nu,\mu^{n}_{\nu}}_{f}(h),g\rangle_{A^{2}_{\mu^{n}_{\nu}}}, \quad g\in H^\infty,\quad h\in H^\infty,
	\end{split}
	\end{equation}
where
	$$
	S^{\nu,\mu^{n}_{\nu}}_{f}(h)(z)
	=\int_{\mathbb{D}}f(\zeta)\overline{h(\zeta)B_{z}^{\mu^{n}_{\nu}}(\zeta)}\nu(\zeta)\,dA(\zeta),
	\quad z\in\mathbb{D},\quad h\in H^\infty.
	$$
Therefore \cite[Theorem 3]{PelRat2020}, our assumption on the boundedness of $h_{f}^{\nu}:A^{p}_{\omega}\to A^{1}_{\nu}$, and \cite[Proposition~5]{PRS1} yield
	\begin{equation}\label{Ie:1-theorem3}
	\begin{split}
	|\langle S^{\nu,\mu^{n}_{\nu}}_{f}(h),g\rangle_{A^{2}_{\mu^{n}_{\nu}}}|
	&=|\langle h_{f}^{\nu}(g),h\rangle_{A^{2}_{\nu}}|
	\le\|h_{f}^{\nu}(g)\|_{A^{1}_{\nu}}\|h\|_{\mathcal{B}}\\
	&\lesssim\|h_{f}^{\nu}\|_{A^{p}_{\omega}\to A^{1}_{\nu}}\|g\|_{A^{p}_{\widetilde{\omega}}}\|h\|_{\mathcal{B}}, \quad g\in H^\infty,\quad h\in H^\infty.
	\end{split}
	\end{equation}
Since $\om,\nu\in\DDD$ by the hypotheses, routine calculations based on \cite[Lemma~2.1]{PelSum14}, \eqref{Eq:Dd-characterization}, \eqref{Ieq:newweight1} and \eqref{Ieq:newweight} now show that there exists an $N_3=N_3(\om,\nu,p)\in\N$ such that for each fixed $n\ge N_3$ we have
	\begin{equation}\label{aoao}
	\begin{split}
	\frac{\widehat{\widetilde{\omega}}(r)^\frac1p\left(\int_r^1t\left(\frac{\mu^n_\nu(t)}{\widetilde{\omega}(t)^\frac1p}\right)^{p'}\,dt\right)^\frac1{p'}}{\widehat{\mu^n_\nu}(r)}
	&\asymp\frac{\widehat{\omega}(r)^\frac1p\left(\int_r^1\left(\frac{\widehat{\nu}(t)(1-t)^{n-1+\frac1p}}{\widehat{\omega}(t)^\frac1p}\right)^{p'}\,dt\right)^\frac1{p'}}{\widehat{\nu}(r)(1-r)^{n}}\\
	&\asymp\frac{\widehat{\omega}(r)^\frac1p\frac{\widehat{\nu}(r)(1-r)^{n-1+\frac1p}}{\widehat{\omega}(r)^\frac1p}(1-r)^\frac1{p'}}{\widehat{\nu}(r)(1-r)^{n}}
	\asymp1,\quad r\to1^-.
	\end{split}
	\end{equation}
Therefore $P_{\mu_{\nu}^{n}}: L^{p}_{\widetilde{\omega}}\to A^{p}_{\widetilde{\omega}}$ is bounded by \cite[Theorem 13]{PelRat2020} for each such $n$. Further, the argument used in the proof of \cite[Theorem~6]{PR2014} further implies that $(A^{p'}_{\tau_{[n p']}})^{\star}\simeq A^{p}_{\widetilde{\omega}}$ via the $A^2_{\mu^n_\nu}$-pairing. By applying this duality relation in \eqref{Ie:1-theorem3}, we deduce that
	\begin{equation}\label{Ie:normofS}
	\|S^{\nu,\mu^{n}_{\nu}}_{f}(h)\|_{A^{p'}_{\tau_{[n p']}}}\lesssim\|h_{f}^{\nu}\|_{A^{p}_{\omega}\to A^{1}_{\nu}}\|h\|_{\mathcal{B}},\quad h\in H^\infty.
	\end{equation}
By choosing $h\equiv1$ in \eqref{Ie:normofS} and applying \eqref{id:generalizedreproducing} we obtain
	\begin{align}\label{add-1}
  \int_{\mathbb{D}}\bigg|\frac{d^n}{dz^n}(z^nf(z))\bigg|^{p'}\tau_{[n p']}(z)\,dA(z)
	&=\int_{\mathbb{D}}\left|\int_{\mathbb{D}}f(\zeta)\overline{B_{z}^{\mu_{\nu}^{n}}(\zeta)}\nu(\zeta)\,dA(\zeta)\right|^{p'}\tau_{[n p']}(z)\,dA(z)\nonumber\\
  &=\|S^{\nu,\mu^{n}_{\nu}}_{f}(1)\|_{A^{p'}_{\tau_{[np']}}}^{p'}
	\lesssim\|h_{f}^{\nu}\|_{A^{p}_{\omega}\to A^{1}_{\nu}}^{p'}.
	\end{align}
This shows, in particular, that the measure in \eqref{Eq:theorem3} is finite.
Let $r\in(0,1)$. Observe that
	\begin{align*}
	&\int_{\mathbb{D}}\left|\frac{d^n}{dz^n}(z^nf(z))\right|^{p'}|h_{r}(z)|^{p'}\tau_{[np']}(z)\,dA(z)
	\lesssim\int_{\mathbb{D}}|S^{\nu,\mu^{n}_{\nu}}_{f}(h_{r})(z)|^{p'}\tau_{[np']}(z)\,dA(z)\\
	&\qquad+
	\int_{\mathbb{D}}\left|\frac{d^n}{dz^n}(z^nf(z))\overline{h_{r}(z)}-S^{\nu,\mu^{n}_{\nu}}_{f}(h_{r})(z)\right|^{p'}\tau_{[np']}(z)\,dA(z)
	=I_1(h_{r})+I_2(h_{r}),\quad h\in\B.
	\end{align*}
It follows from \eqref{Ie:normofS} that
	\begin{equation}\label{ngngngn}
	I_1(h_{r})=\|S^{\nu,\mu^{n}_{\nu}}_{f}(h_{r})\|_{A^{p'}_{\tau_{[n p']}}}^{p'}
	\lesssim\|h_{f}^{\nu}\|_{A^{p}_{\omega}\to A^{1}_{\nu}}^{p'}\|h_{r}\|_{\mathcal{B}}^{p'},\quad h\in\B.
	\end{equation}
To get the same upper estimate for $I_2(h_{r})$ we argue as follows. Let first $\varepsilon\in(0,1/p')$ to be fixed later. Then \eqref{id:generalizedreproducing}, \eqref{Ieq:newweight1}, the reproducing formula in $A_{\mu_{\nu}^{n}}^{1}$, Fubini's theorem and H\"older's inequality yield
	\begin{equation}\label{estimateofI2}
	\begin{split}
  I_{2}(h_{r})
	&=\int_{\mathbb{D}}\left|\int_{\mathbb{D}}f(\zeta)\overline{(h_{r}(z)-{h_{r}(\zeta)})B_{z}^{\mu_{\nu}^{n}}(\zeta)}\nu(\zeta)\,dA(\zeta)\right|^{p'}
	\tau_{[np']}(z)\,dA(z)\\
&=\int_{\mathbb{D}}\left|\int_{\mathbb{D}}f(\zeta)\overline{\int_{\mathbb{D}}(h_{r}(z)-h_{r}(u))B_{z}^{\mu_{\nu}^{n}}(u)
\overline{B_{\zeta}^{\mu_{\nu}^{n}}(u)}\mu_{\nu}^{n}(u)dA(u)}\nu(\zeta)\,dA(\zeta)\right|^{p'}\tau_{[np']}(z)\,dA(z)\\
&=\int_{\mathbb{D}}\left|\int_{\mathbb{D}}\frac{d^{n}}{du^{n}}(u^{n}f(u))\overline{(h_{r}(z)-h_{r}(u))B_{z}^{\mu_{\nu}^{n}}(u)}
\mu_{\nu}^{n}(u)dA(u)\right|^{p'}\tau_{[np']}(z)\,dA(z)\\
	&\lesssim\int_{\mathbb{D}}\left(\int_{\mathbb{D}}
	\left|\frac{d^{n}}{du^{n}}(u^{n}f(u))(h_{r}(z)-h_{r}(u))B_{z}^{\mu_{\nu}^{n}}(u)\right|\widetilde{\nu}_{[n]}(u)\,dA(u)\right)^{p'}\tau_{[np']}(z)\,dA(z)\\
	&\le\int_{\mathbb{D}}
	\left(\int_{\mathbb{D}}
\bigg|\frac{d^{n}}{d\zeta^{n}}(\zeta^{n}f(\zeta))\bigg|^{p'}|B_{z}^{\mu_{\nu}^{n}}(\zeta)|
\tau_{[\e p'+np']}(\zeta)\,dA(\zeta)\right)\\
	&\quad\cdot\left(\int_{\mathbb{D}}|h_{r}(z)-h_{r}(\zeta)|^{p}|B_{z}^{\mu_{\nu}^{n}}(\zeta)|
	\frac{\widetilde{\nu}(\zeta)^{p}}{\left(\tau_{[\e p']}(\zeta)\right)^{\frac{p}{p'}}}\,dA(\zeta)\right)^{\frac{p'}{p}}\tau_{[np']}(z)\,dA(z),\quad h\in\B.
  \end{split}
	\end{equation}
 We claim that, for all $n$ sufficiently large depending on $\om$, $\nu$ and $p$, we have
	\begin{equation}\label{finalestimate}
  \int_{\mathbb{D}}|h(z)-h(\zeta)|^{p}|B_{z}^{\mu_{\nu}^{n}}(\zeta)|
	\frac{\widetilde{\nu}(\zeta)^{p}}{\left(\tau_{[\e p']}(\zeta)\right)^{\frac{p}{p'}}}\,dA(\zeta)
	\lesssim\|h\|_{\mathcal{B}}^{p}\frac{\widehat{\omega}(z)}{\widehat{\nu}_{[n+\e p]}(z)},\quad h\in\B,\quad z\in\mathbb{D},
	\end{equation}
the proof of which will be provided later. Therefore, by \eqref{estimateofI2}, \eqref{finalestimate}, Fubini's theorem, \cite[Theorem~1]{PR2014}, \eqref{Ieq:newweight} and \eqref{add-1}, we deduce
	\begin{equation*}
	\begin{split}
	I_{2}(h_{r})
	&\lesssim\|h_{r}\|_{\mathcal{B}}^{p'}
	\int_{\mathbb{D}}\left(\int_{\mathbb{D}}\bigg|\frac{d^{n}}{d\zeta^{n}}(\zeta^{n}f(\zeta))\bigg|^{p'}|B_{z}^{\mu_{\nu}^{n}}(\zeta)|\tau_{[\e p'+np']}(\zeta)\,dA(\zeta)\right)\\
	&\quad\cdot\left(\frac{\widehat{\omega}(z)}{\widehat{\nu}_{[n+\e p]}(z)}\right)^{\frac{p'}{p}}\tau_{[np']}(z)\,dA(z)\\
	&=\|h_{r}\|_{\mathcal{B}}^{p'}\int_{\mathbb{D}}\bigg|\frac{d^{n}}{d\zeta^{n}}(\zeta^{n}f(\zeta))\bigg|^{p'}\tau_{[\varepsilon p'+np']}(\zeta)\left(\int_{\mathbb{D}}|B_{z}^{\mu_{\nu}^{n}}(\zeta)|(1-|z|)^{-\varepsilon p'+n}\widetilde{\nu}(z)\,dA(z)\right)dA(\zeta)\\
	&\asymp\|h_{r}\|_{\mathcal{B}}^{p'}\int_{\mathbb{D}}\bigg|\frac{d^{n}}{d\zeta^{n}}(\zeta^{n}f(\zeta))\bigg|^{p'}\tau_{[\varepsilon p'+np']}(\zeta)
	\left(\int_{0}^{|\zeta|}\frac{(1-t)^{-\varepsilon p'+n+1}\widetilde{\nu}(t)}{\widehat{\mu_{\nu}^{n}}(t)(1-t)}\,dt+1\right)\,dA(\zeta)\\
	&\lesssim\|h_{r}\|_{\mathcal{B}}^{p'}\bigg\|\frac{d^{n}}{d(\cdot)^{n}}\big((\cdot)^{n}f(\cdot)\big)\bigg\|_{A^{p'}_{\tau_{[np']}}}^{p'}
\lesssim\|h_{r}\|_{\mathcal{B}}^{p'}\|h_{f}^{\nu}\|_{A^{p}_{\om}\to A^{1}_{\nu}}^{p'},\quad h\in\B.
	\end{split}
	\end{equation*}
This estimate together with \eqref{ngngngn} gives
\[
\int_{\mathbb{D}}\left|\frac{d^n}{dz^n}(z^nf(z))\right|^{p'}|h_{r}(z)|^{p'}\tau_{[np']}(z)\,dA(z)
	\lesssim\|h_{r}\|_{\mathcal{B}}^{p'}\|h_{f}^{\nu}\|_{A^{p}_{\om}\to A^{1}_{\nu}}^{p'},\quad h\in\B,
\]
where the constant of comparison in the formula above is independent on the choice of $r$.
This together with Fatou's lemma and the fact that $\sup\limits_{r\in(0,1)}\|h_{r}\|_{\mathcal{B}}\leq \|h\|_{\mathcal{B}}$ implies that
\begin{align*}
\bigg\|h\frac{d^{n}}{d(\cdot)^{n}}\big((\cdot)^{n}f(\cdot)\big) \bigg\|_{A_{\tau_{[np']}}^{p'}}^{p'}
&\leq\liminf_{r\to 1^{-}}\bigg\|h_{r}\frac{d^{n}}{d(\cdot)^{n}}\big((\cdot)^{n}f(\cdot)\big) \bigg\|_{A_{\tau_{[np']}}^{p'}}^{p'}\\
&\lesssim\liminf_{r\to 1^{-}}\|h_{r}\|_{\mathcal{B}}^{p'}\|h_{f}^{\nu}\|_{A^{p}_{\om}\to A^{1}_{\nu}}^{p'}
	\leq\|h\|_{\mathcal{B}}^{p'}\|h_{f}^{\nu}\|_{A^{p}_{\om}\to A^{1}_{\nu}}^{p'},\quad h\in\B,
\end{align*}
which shows that \eqref{aimoftheorem3} holds. An application of Lemma~\ref{lemma:bloch-measure} together with \eqref{aimoftheorem3} now yields the desired norm estimate from below.

Therefore, to complete the proof, it remains to show that \eqref{finalestimate} is valid. Lemma~\ref{lemma:weights-simple} shows that we may choose $c=c(\om,p)\in(0,1)$ such that $\widetilde{\om}_{[2cp-2p]}\in\DDD$ because $\om\in\DDD$ by the hypothesis. Further, it follows from \cite[Lemma~2.1]{PelSum14} that we may fix a $k=k(\om,p)\in\N$ such that
	\begin{equation}\label{Ie1}
	\int_{\mathbb{D}}\frac{\widetilde{\omega}(\zeta)(1-|\zeta|)^{-2p+2cp}}
	{|1-\overline{\zeta}z|^{2(k-(2-2c)p)}}\,dA(\zeta)
	\asymp\frac{\whw(z)}{(1-|z|)^{2k+2cp-2p-1}},\quad z\in\mathbb{D}.
	\end{equation}
Likewise, by Lemma~\ref{lemma:weights-simple}, we may choose $\varepsilon=\varepsilon(\om,p)\in(0,1)$ such that $\widetilde{\om}_{[(-2\varepsilon p)]}\in\DDD$. Then using Lemma~\ref{weightedkernel}, we find a natural number $N_4=N_4(\om,\nu,p)\ge N_3$ such that for each fixed $n\ge N_4$ we have
	\begin{equation}\label{Ie2}
  \begin{split}
  \int_{\mathbb{D}}\left(|1-\zeta\bar{z}|^{k}|B_{z}^{\mu_{\nu}^{n}}(\zeta)|\right)^{2}\widetilde{\omega}(\zeta)
	(1-|\zeta|)^{-2\varepsilon p}\,dA(\zeta)
	&\lesssim\int_{0}^{|z|}\frac{\widehat{\widetilde{\omega}}(t)(1-t)^{-2\varepsilon p}}
{\widehat{\mu_{\nu}^{n}}(t)^{2}(1-t)^{2-2k}}\,dt+1\\
&\asymp\frac{\widehat{\omega}(z)}{\widehat{\nu}(z)^{2}(1-|z|)^{1-2k+2n+2\varepsilon p}},\quad z\in\D.
  \end{split}
	\end{equation}
In addition, a direct application of \cite[Proposition~2.4]{WL} or \cite[Theorem 5.8]{zhu} gives
	\begin{equation}\label{Ie3}
  \frac{(1-|z|)^{1-c}(1-|\zeta|)^{1-c}|h(z)-h(\zeta)|}{|z-\zeta||1-\overline{\zeta}z|^{1-2c}}\lesssim\|h\|_{\mathcal{B}},\quad h\in\B,
	\end{equation}
for all $z,\zeta\in\D$ with $z\ne\zeta$. Thus, by \eqref{Ie1}, \eqref{Ie2}, \eqref{Ie3}, and the Cauchy-Schwarz inequality, we finally deduce
	\begin{align*}
    &\quad \int_{\mathbb{D}}|h(z)-h(\zeta)|^{p}|B_{z}^{\mu_{\nu}^{n}}(\zeta)|(1-|\zeta|)^{-\varepsilon p}
	\frac{\widetilde{\nu}(\zeta)^{p}}{\tau(\zeta)^{\frac{p}{p'}}}\,dA(\zeta)\\
	&=\int_{\mathbb{D}}|h(z)-h(\zeta)|^{p}|B_{z}^{\mu_{\nu}^{n}}(\zeta)|(1-|\zeta|)^{-\varepsilon p}\widetilde{\omega}(\zeta)\,dA(\zeta)\\
	&\lesssim\|h\|_{\mathcal{B}}^{p}\int_{\mathbb{D}}\frac{|z-\zeta|^{p}|1-\zeta\overline{z}|^{(1-2c)p}}{(1-|z|)^{(1-c)p}(1-|\zeta|)^{(1-c)p}}
	|B_{z}^{\mu_{\nu}^{n}}(\zeta)|(1-|\zeta|)^{-\varepsilon p}\widetilde{\omega}(\zeta)\,dA(\zeta)\\
	&\le\|h\|_{\mathcal{B}}^{p}(1-|z|)^{(c-1)p}\int_{\mathbb{D}}|1-\zeta\overline{z}|^{k}|1-\zeta\overline{z}|^{(2-2c)p-k}|B_{z}^{\mu_{\nu}^{n}}(\zeta)|
	(1-|\zeta|)^{-\varepsilon p-p+cp}\widetilde{\omega}(\zeta)\,dA(\zeta)\\
	&\le\|h\|_{\mathcal{B}}^{p}(1-|z|)^{(c-1)p}\left(\int_{\mathbb{D}}\left(|1-\zeta\overline{z}|^{k}|B_{z}^{\mu_{\nu}^{n}}(\zeta)|\right)^{2}
	\frac{\widetilde{\omega}(\zeta)}
	{(1-|\zeta|)^{2\varepsilon p}}\,dA(\zeta)\right)^{\frac{1}{2}}\\
	&\quad\cdot\left(\int_{\mathbb{D}}\frac{\widetilde{\omega}(\zeta)(1-|\zeta|)^{-2p+2cp}}
	{|1-\zeta\overline{z}|^{2(k-(2-2c)p)}}dA(\zeta)\right)^{\frac{1}{2}}\\
	&\lesssim\|h\|_{\mathcal{B}}^{p}(1-|z|)^{(c-1)p}
	\left(\frac{\widehat{\omega}(z)}{\widehat{\nu}(z)^{2}(1-|z|)^{1-2k+2n+2\varepsilon p}}\right)^{\frac12}
	\left(\frac{\whw(z)}{(1-|z|)^{2k+2cp-2p-1}}\right)^{\frac12}\\
	&=\|h\|_{\mathcal{B}}^{p}(1-|z|)^{-n-\varepsilon p}\frac{\widehat{\omega}(z)}{\widehat{\nu}(z)},\quad h\in\B,\quad z\in\D.
	\end{align*}
This completes the proof of \eqref{finalestimate}. We finish the proof of the theorem by noting that the number $N=N(\om,\nu,p)\in\N$ appearing in the statement equals to $\max\{N_2,N_4\}$.\hfill$\Box$

\section{Proof of Theorem~\ref{theorem3}}

\noindent\emph{Proof of sufficiency}. Assume that \eqref{Eq:Thm3-hypothesis} is satisfied. Since $f\in A^1_{\nu_{\log}}$ and $\nu\in\DDD$ by the hypotheses, \eqref{Id: newrepresentation}, Fubini's theorem and \cite[Theorem~1]{PR2014} yield
	\begin{equation}\label{aaslkjfh}
	\begin{split}
	\|h^\nu_f(g)\|_{A^1_\nu}
	&=\int_\D\left|\int_{\D}\overline{g(\z)B_z^{\nu}(\z)}\left(\frac{d^n}{d\z^n}(\z^nf(\z))\right)\mu^n_{\nu}(\z)\,dA(\z)\right|\nu(z)\,dA(z)\\
	&\leq \int_{\D}|g(\z)|\left|\frac{d^n}{d\z^n}(\z^nf(\z))\right|\mu^n_\nu(\z)\left(\int_\D|B^\nu_{\z}(z)|\nu(z)\,dA(z)\right)dA(\z)\\
	&\lesssim\int_{\D}|g(\z)|\left|\frac{d^n}{d\z^n}(\z^nf(\z))\right|\mu^n_\nu(\z)\log\frac{e}{1-|\z|}\,dA(\z)\\
	&=\int_{\D}|g(\z)|\,d\eta^n_{f,\nu}(\z), \quad g\in H^\infty,
	\end{split}
	\end{equation}
where
	$$
	d\eta^n_{f,\nu}(\z)=\left|\frac{d^n}{d\z^n}(\z^nf(\z))\right|\mu^n_\nu(\z)\log\frac{e}{1-|\z|}\,dA(\z).
	$$
Therefore a sufficient condition for $h^\nu_f:A^p_\om\to A^1_\nu$ to be bounded is that $\eta^n_{f,\nu}$ is a $1$-Carleson measure for $A^p_\om$. By \cite[Theorem~2]{LiuRattya} this is equivalent to \eqref{tvtvt}, where $\mu$ is replaced by $\eta^n_{f,\nu}$ and $q=1$. To obtain this, we first observe that since $\om\in\DDD$ by the hypothesis, there exists an $r_0=r_0(\om)\in(0,1)$ such that $\om(\Delta(z,r))\asymp\widehat{\om}(z)(1-|z|)$ as $|z|\to1^-$, provided $r\in(r_0,1)$. This asymptotic equality together with the assumption \eqref{Eq:Thm3-hypothesis}, \eqref{Ieq:newweight1} and \cite[Lemma~2.1]{PelSum14} yield
	\begin{align*}
	\frac{\eta^n_{f,\nu}(\Delta(z,r))}{\om(\Delta(z,r))^{\frac1p}}
	&\asymp\frac{\int_{\Delta(z,r)}\left|\frac{d^n}{d\z^n}(\z^nf(\z))\right|\mu_\nu^n(\z)\log\frac{e}{1-|\z|}\,dA(\z)}
	{\widehat{\om}(z)^{\frac1p}(1-|z|)^{\frac1p}}\\
	&\lesssim\frac{\int_{\Delta(z,r)}\widehat{\om}(\z)^{\frac1p}(1-|\z|)^{\frac{1}{p}-2}\,dA(\z)}{\widehat{\om}(z)^{\frac1p}(1-|z|)^{\frac1p}}\asymp1,\quad |z|\to1^-,
	\end{align*}
and it follows from the density argument that $h^\nu_f:A^p_\om\to A^1_\nu$ is bounded, and
	$$
	\|h^\nu_{f}\|_{A^p_\om\rightarrow A^1_\nu}
	\lesssim\sup_{z\in\D}\left|\frac{d^n}{dz^n}(z^nf(z))\right|(1-|z|)^n
	\frac{\widehat{\nu}(z)(1-|z|)}{\left(\whw(z)(1-|z|)\right)^{\frac{1}{p}}}\log\frac{e}{1-|z|}.
	$$
An application of Lemma~\ref{Lemma:derivative2} completes the proof, provided $n\in\N$ is sufficiently large, say $n\ge N_1=N_1(\om,\nu,p,q)\in\N$. \hfill$\Box$

\medskip

\medskip

\noindent\emph{Proof of necessity}. Assume that $h^\nu_f:A^p_\om\to A^1_\nu$ is bounded. Since $\nu\in\DDD$ by the hypothesis, $(A^1_{\nu})^\star\simeq\B$ via the $A^2_\nu$-pairing by \cite[Theorem~3]{PelRat2020}. By \eqref{Id: newrepresentation} and Fubini's theorem, we have
	\begin{align*}
  \left<h^\nu_f(g),h\right>_{A^2_\nu}
	&=\int_\D\left(\int_\D\left(\frac{d^n}{d\z^n}(\z^nf(\z))\right)\overline{g(\z)}\overline{B^\nu_z(\z)}\mu^n_\nu(\z)\,dA(\z)\right)
	\overline{h(z)}\nu(z)\,dA(z)\\
  &=\int_\D\left(\frac{d^n}{d\z^n}(\z^nf(\z))\right)\overline{g(\z)}\mu^n_\nu(\z)
	\left(\overline{\int_\D h(z)\overline{B^\nu_\z(z)}\nu(z)\,dA(z)}\right)\,dA(\z)\\
  &=\int_\D\left(\frac{d^n}{d\z^n}(\z^nf(\z))\right)\overline{g(\z)}\overline{h(\z)}\mu^n_\nu(\z)\,dA(\z),\quad g,h\in H^\infty,
	\end{align*}
and therefore
	\begin{equation}\label{dualq=1}
	\begin{split}
  \left|\int_\D\left(\frac{d^n}{d\z^n}(\z^nf(\z))\right)\overline{g(\z)}\overline{h(\z)}\mu^n_\nu(\z)\,dA(\z)\right|
	\le\|h^\nu_f(g)\|_{A^1_\nu}\|h\|_\B\le\|h^\nu_f\|_{A^p_\om\rightarrow A^1_\nu}\|g\|_{A^p_\om}\|h\|_\B
	\end{split}
	\end{equation}
for all $g,h\in H^\infty$. Choose now $g(\z)=g_{z,n,\nu}(\z)=B_z^{\mu_\nu^n}(\z)$ and $h(\zeta)=h_z(\z)=1-\log(1-\overline{z}\z)$ for all $z,\z\in\D$. Then, obviously $\|h_z\|_\B\lesssim1$ for all $z\in\D$. Further, by \cite[Theorem~1]{PR2014}, \eqref{Ieq:newweight} and \cite[Lemma~2.1]{PelSum14} there exists an $N_2=N_2(\om,\nu,p)$ such that
	\begin{equation}\label{thmlkj}
  \|g_{z,n,\nu}\|_{A^p_\om}^p
	\asymp\int_0^{|z|}\frac{\whw(t)}{\widehat{\nu}(t)^p(1-t)^{np+p}}\,dt+1
	\asymp\frac{\whw(z)}{\widehat{\nu}(z)^p(1-|z|)^{np+p-1}},\quad z\in\D,
	\end{equation}
for each $n\ge N_2$. Therefore \eqref{dualq=1} yields
	\begin{equation*}
	\begin{split}
  \left|\int_\D\left(\frac{d^n}{d\z^n}(\z^nf(\z))\right)\overline{B_z^{\mu_\nu^n}(\z)}\log\frac{e}{1-\overline{\z}z}\mu^n_\nu(\z)\,dA(\z)\right|
	\lesssim\|h^\nu_f\|_{A^p_\om\rightarrow A^1_\nu}\frac{\left(\whw(z)(1-|z|)\right)^\frac1p}{\widehat{\nu}(z)(1-|z|)^{n+1}},\quad z\in\D.
	\end{split}
	\end{equation*}
In view of the identity
	$$
	\log\frac{e}{1-|z|}=\log\frac{1-\overline{\zeta}z}{1-|z|}+\log\frac{e}{1-\overline{\z}z}
	$$
and Lemma~\ref{Lemma:derivative2}, it remains to prove the same upper bound for the term
	$$
	I(z)=\left|\int_\D\left(\frac{d^n}{d\z^n}(\z^nf(\z))\right)\overline{B_z^{\mu_\nu^n}(\z)}\log\frac{1-\overline{\zeta}z}{1-|z|}\mu^n_\nu(\z)\,dA(\z)\right|.
	$$
To do this, observe that
	\begin{equation*}
	\begin{split}
  \left|\left\langle \frac{d^n}{d(\cdot)^n}((\cdot)^nf(\cdot)),g\right\rangle_{A^2_{\mu^n_\nu}}\right|
	&\asymp|h^\nu_f(g)(0)|
	\lesssim\|h^\nu_f(g)\|_{A^1_\nu}
	\le\|h^\nu_f\|_{A^p_\om\rightarrow A^1_\nu}\|g\|_{A^p_\om},\quad g\in H^\infty,
	\end{split}
	\end{equation*}
by \eqref{Id: newrepresentation}. An application of this inequality to the function
	$$
	g(\zeta)=g_{z,n,\nu}(\z)=B_z^{\mu_\nu^n}(\z)\log\frac{1-\overline{z}\z}{1-|z|},\quad \z\in\D,
	$$
gives
	\begin{equation*}
	\begin{split}
  I(z)\lesssim\|h^\nu_f\|_{A^p_\om\rightarrow A^1_\nu}
	\left(\int_{\D}\left|B_z^{\mu_\nu^n}(\z)\log\frac{1-\overline{z}\z}{1-|z|}\right|^p\om(\zeta)\,dA(\zeta)\right)^\frac1p.
	\end{split}
	\end{equation*}
The inequality $\log x\le x-1$, applied to $x=\left(\frac{|1-\overline{z}\z|}{1-|z|}\right)^\e$, implies
	\begin{equation*}
	\begin{split}
  I(z)^p&\lesssim\|h^\nu_f\|^p_{A^p_\om\rightarrow A^1_\nu}\left(
	\frac{1}{(1-|z|)^{\e p}}\int_{\D}\left|(1-\overline{z}\z)^\e B_z^{\mu_\nu^n}(\z)\right|^p\om(\zeta)\,dA(\zeta)
	+\|B_z^{\mu_\nu^n}\|_{A^p_\om}^p\right)\\
	&\lesssim\frac{\|h^\nu_f\|^p_{A^p_\om\rightarrow A^1_\nu}}{(1-|z|)^{\e p}}\int_{\D}\left|(1-\overline{z}\z)^\e B_z^{\mu_\nu^n}(\z)\right|^p\om(\zeta)\,dA(\zeta)
	\end{split}
	\end{equation*}
for each fixed $0<\e<1$. By Lemma~\ref{weightedkernel}, H\"older's inequality, \cite[Theorem~1]{PR2014} and \cite[Lemma~2.1]{PelSum14}, and by choosing $N_3=N_3(\om,\nu,p)\in\N$ sufficiently large we deduce
	\begin{equation*}
	\begin{split}
  \int_{\D}\left|(1-\overline{z}\z)^\e B_z^{\mu_\nu^n}(\z)\right|^p\om(\zeta)\,dA(\zeta)
	&=\int_{\D}\left|(1-\overline{z}\z)B_z^{\mu_\nu^n}(\z)\right|^{\e p}\left|B_z^{\mu_\nu^n}(\z)\right|^{p(1-\e)}\om(\zeta)\,dA(\zeta)\\
	&\le\left(\int_{\D}\left|(1-\overline{z}\z)B_z^{\mu_\nu^n}(\z)\right|^2\om(\zeta)\,dA(\zeta)\right)^\frac{\e p}{2}\\
	&\quad\cdot\left(\int_{\D}\left|B_z^{\mu_\nu^n}(\z)\right|^{\frac{2p(1-\e)}{2-\e p}}\om(\zeta)\,dA(\zeta)\right)^{\frac{2-\e p}{2}}\\
	&\lesssim\left(\int_0^{|z|}\frac{\widehat{\om}(t)}{(\widehat{\nu}(t)(1-t)^n)^2}\,dt+1\right)^\frac{\e p}{2}\\
	&\quad\cdot\left(\int_0^{|z|}\frac{\widehat{\om}(t)}
	{(\widehat{\nu}(t)(1-t)^n)^{\frac{2p(1-\e)}{2-\e p}}(1-t)^{\frac{2p(1-\e)}{2-\e p}}}\,dt+1\right)^{\frac{2-\e p}{2}}\\
	&\asymp\left(\frac{\widehat{\om}(z)}{\widehat{\nu}(z)^2(1-|z|)^{2n-1}}\right)^\frac{\e p}{2}\\
	&\quad\cdot\left(\frac{\widehat{\om}(z)}
	{\widehat{\nu}(z)^{\frac{2p(1-\e)}{2-\e p}}(1-|z|)^{(n+1)\frac{2p(1-\e)}{2-\e p}-1}}\right)^{\frac{2-\e p}{2}}\\
	&=\frac{\widehat{\om}(z)(1-|z|)}{\widehat{\nu}(z)^{p}(1-|z|)^{p(n+1)-\e p}}
	\end{split}
	\end{equation*}
for each fixed $n\ge N_3$, and hence
	\begin{equation*}
	\begin{split}
  I(z)&\lesssim\|h^\nu_f\|_{A^p_\om\rightarrow A^1_\nu}\frac{(\widehat{\om}(z)(1-|z|))^\frac1p}{\widehat{\nu}(z)(1-|z|)^{n+1}}.
	\end{split}
	\end{equation*}
We finish the proof of the theorem by noting that the number $N=N(\om,\nu,p,q)\in\N$ appearing in the statement equals to $\max\{N_1,N_2,N_3\}$. \hfill$\Box$

\section{Proofs of Theorems~\ref{theorem4} and~\ref{theorem5}}

\medskip

\noindent\emph{Proof of Theorem~\ref{theorem4}}.
Assume first \eqref{result4}. Let $b$ be a positive function to be fixed later. By \cite[Proposition~5]{PRS1}, H\"older's inequality and Fubini's theorem we deduce
		\begin{align*}
		\|h^\nu_f(g)\|^q_{A^q_\nu}
		&\asymp\|h^\nu_f(g)\|^q_{A^q_{\widetilde{\nu}}}
		\le\int_\D\left(\int_{\D}|{g(\z)B_z^{\nu}(\z)}|\left|\frac{d^n}{d\z^n}(\z^nf(\z))\right|\mu^n_{\nu}(\z)\,dA(\z)\right)^qb(z)^q\frac{\widetilde{\nu}(z)}{b(z)^q}\,dA(z)\\
		&\leq \left(\int_\D\int_{\D}|{g(\z)B_z^{\nu}(\z)}|\left|\frac{d^n}{d\z^n}(\z^nf(\z))\right|\mu^n_{\nu}(\z)\,dA(\z)b(z)\widetilde{\nu}(z)\,dA(z)\right)^q\\
		&\quad\cdot\left(\int_\D\frac{\widetilde{\nu}(z)}{b(z)^{\frac{q}{1-q}}}\,dA(z)\right)^{1-q}\\
		&=\left(\int_\D|g(\z)|\left|\frac{d^n}{d\z^n}(\z^nf(\z))\right|\mu^n_\nu(\z)\int_\D|B^\nu_\z(z)|b(z)\widetilde{\nu}(z)\,d A(z)\,dA(\z)\right)^q\\
		&\quad\cdot\left(\int_\D\frac{\widetilde{\nu}(z)}{b(z)^{\frac{q}{1-q}}}\,dA(z)\right)^{1-q},\quad g\in H^\infty,
		\end{align*}
for all $n\in\N\cup\{0\}$. Now, by \cite[Lemma~2.1]{PelSum14} and \eqref{Eq:Dd-characterization}, we may choose $\varepsilon=\varepsilon(\nu,q)>0$ sufficiently small such that $z\mapsto\frac{\widehat{\nu}(z)^{\frac1q}}{(1-|z|)^{1+\varepsilon}}$ belongs to $\DDD$ and $z\mapsto\frac{\widehat{\nu}(z)^{\frac1q-1}}{(1-|z|)^{1+\varepsilon}}$ is integrable over $\D$ with respect to the Lebesgue area measure. Take now $b(z)=\widehat{\nu}(z)^{\frac1q-1}(1-|z|)^{-\varepsilon}$. Then
		\begin{equation*}\label{eq: secondfactor}
		\int_\D\frac{\widetilde{\nu}(z)}{b(z)^{\frac{q}{1-q}}}\,dA(z)=\int_\D(1-|z|)^{\frac{q}{1-q}\varepsilon-1}\,dA(z)<\infty.
		\end{equation*}
This together with \eqref{Ieq:newweight1} and \cite[Theorem~1]{PR2014} gives
		\begin{equation*}
    \begin{split}
		\|h^\nu_f(g)\|_{A^q_\nu}
		&\lesssim\int_\D|g(\z)|\left|\frac{d^n}{d\z^n}(\z^nf(\z))\right|\widehat{\nu}(\z)(1-|\z|)^{n-1}
		\left(\int_\D|B^\nu_\z(z)|\frac{\widehat{\nu}(z)^{\frac1q}}{(1-|z|)^{1+\varepsilon}}\,dA(z)\right)\,dA(\z)\\
		&\lesssim\int_\D|g(\z)|\left|\frac{d^n}{d\z^n}(\z^nf(\z))\right|\widehat{\nu}(\z)(1-|\z|)^{n-1}\left(\int_0^{|\z|}\frac{\widehat{\nu}(t)^{\frac1q-1}}{(1-t)^{1+\varepsilon}}\,dt+1\right)\,dA(\z)\\
		&\lesssim\int_\D|g(\z)|\left|\frac{d^n}{d\z^n}(\z^nf(\z))\right|\widehat{\nu}(\z)(1-|\z|)^{n-1}\,dA(\z),\quad g\in H^\infty.
    \end{split}
		\end{equation*}
H\"older's inequality and \cite[Proposition~5]{PRS1} now yield
\begin{align*}
  \|h^\nu_f(g)\|_{A^q_\nu}&\lesssim\left(\int_\D|g(\z)|^p\widetilde{\om}(\z)\,dA(\z)\right)^{\frac1p}\cdot \left(\int_\D\left|\frac{d^n}{d\z^n}(\z^nf(\z))\right|^{p'}\widehat{\nu}(\z)^{p'}(1-|\z|)^{(n-1)p'}\widetilde{\om}(\z)^{-\frac{p'}{p}}\,dA(\z)\right)^{\frac{1}{p'}}\\
  &\asymp\|g\|_{A^p_\om}\left(\int_\D\left|\frac{d^n}{d\z^n}(\z^nf(\z))\right|^{p'}(1-|\z|)^{np'}
		\left(\frac{\widehat{\nu}(\z)}{\whw(\z)}\right)^{p'}\frac{\widehat{\om}(\z)}{1-|\z|}\, dA(\z)\right)^{\frac{1}{p'}}.
\end{align*}
This together with Lemma~\ref{Lemma:derivative}, the hypothesis \eqref{result4} and the density argument shows that $h^\nu_f:A^p_{\om}\rightarrow A^q_{\nu}$ is bounded with the desired upper estimate for the operator norm, provided $n$ is sufficiently large, say $n\ge N_1=N_1(\om,\nu,p,q)\in\N$.
	
Conversely, assume that $h^\nu_f:A^p_{\om}\rightarrow A^q_{\nu}$ is bounded. Then \eqref{Id: newrepresentation}, the subharmonicity of $|h^\nu_f(g)|$ and \cite[Proposition~5]{PRS1} give
	\begin{equation}\label{gsgsgsg}
	\begin{split}
  \left|\int_\D\left(\frac{d^n}{d\z^n}(\z^nf(\z))\right)\overline{g(\z)}\mu^n_{\nu}(\z)\,dA(\z)\right|
	&\asymp|h^\nu_f(g)(0)|
	\lesssim\|h^\nu_f(g)\|_{A^q_\nu}\\
	&\le\|h^\nu_f\|_{A^p_\om\rightarrow A^q_\nu}\|g\|_{A^p_\om}
	\asymp\|h^\nu_f\|_{A^p_\om\rightarrow A^q_\nu}\|g\|_{A^p_{\widetilde{\om}}}.
	\end{split}
	\end{equation}
Now \cite[Theorem~13]{PelRat2020} implies that $P_{\mu^n_\nu}:L^p_{\widetilde{\om}}\to A^p_{\widetilde{\om}}$ is bounded, provided $n$ is sufficiently large, say $n\ge N_2=N_2(\om,\nu,p)\in\N$. Hence $(A^p_{\widetilde{\om}})^\star\simeq A^{p'}_{W_{p'}}$, where $W_{p'}=(\frac{\mu^n_\nu}{\widetilde{\om}})^{p'}\widetilde{\om}$, via the $A^2_{\mu^n_\nu}$-pairing by \cite[Theorem~6]{PR2014}. Therefore
	\begin{align*}
  \|h^\nu_f\|^{p'}_{A^p_\om\rightarrow A^q_\nu}
	&\gtrsim\int_{\D}\left|\frac{d^n}{d\zeta^n}(\zeta^nf(\zeta))\right|^{p'}W_{p'}(\z)\,dA(\z)\\
  &\asymp\int_\D\left|\frac{d^n}{d\zeta^n}(\zeta^nf(\zeta))\right|^{p'}(1-|\z|)^{np'}
		\left(\frac{\widehat{\nu}(\z)}{\whw(\z)}\right)^{p'}\frac{\widehat{\om}(\z)}{1-|\z|}\,dA(\z).
	\end{align*}
An application of Lemma~\ref{Lemma:derivative} finishes the proof of the necessity for $n\in\N$ is sufficiently large, say $n\ge N_3=N_3(\om,\nu,p)\ge N_2$. The number $N=N(\om,\nu,p)\in\N$ appearing in the statement of the theorem equals to $\max\{N_1,N_3\}$. \hfill$\Box$

\medskip

\noindent\emph{Proof of Theorem \ref{theorem5}}. Assume first \eqref{result5}. By the proof of Theorem~\ref{theorem4}, Lemma~\ref{Lemma:derivative2} and \cite[Theorem~1]{PR2015} we deduce
		\begin{equation*}
    \begin{split}
		\|h^\nu_f(g)\|_{A^q_\nu}
		&\lesssim\int_\D|g(\z)|\left|\frac{d^n}{d\z^n}(\z^nf(\z))\right|\widehat{\nu}(\z)(1-|\z|)^{n-1}\,dA(\z)\\
		&\lesssim\int_\D|g(\z)|\frac{(\whw(\z)(1-|\z|))^{\frac1p}}{(1-|\z|)^2}\,dA(\z)\lesssim\|g\|_{A^p_{\om}},\quad g\in H^\infty,
    \end{split}
		\end{equation*}
for each $n\in\N$ is sufficiently large, say $n\ge N_1=N_1(\om,\nu,p)\in\N$. This reasoning together with the density argument shows that $h^\nu_f: A^p_{\om}\rightarrow A^q_{\nu}$ is bounded with the desired upper estimate for the operator norm.

Conversely, assume that $h^\nu_f: A^p_{\om}\rightarrow A^q_{\nu}$ is bounded. By choosing $g(\zeta)=g_{z,n,\nu}(\z)=B^{\mu^n_\nu}_z(\z)$ in \eqref{gsgsgsg}, \eqref{thmlkj} yields
	$$
	\left|\frac{d^n}{dz^n}(z^nf(z))\right|
	\lesssim\|h^\nu_f\|_{A^p_\om\rightarrow A^q_\nu}\|g_{z,n,\nu}\|_{A^p_\om}
	\asymp\|h^\nu_f\|_{A^p_\om\rightarrow A^q_\nu}\frac{(\whw(z)(1-|z|))^{\frac1p}}{\widehat{\nu}(z)(1-|z|)^{n+1}},\quad z\in\D,
	$$
provided $n\in\N$ is sufficiently large, say $n\ge N_2=N_2(\om,\nu,p)\in\N$. An application of Lemma~\ref{Lemma:derivative2} finishes the proof of the theorem. The number $N=N(\om,\nu,p)\in\N$ appearing in the statement of the theorem equals to $\max\{N_1,N_2\}$. \hfill$\Box$

\end{document}